\documentclass[final]{siamart190516}

\usepackage[utf8]{inputenc}
\usepackage[T1]{fontenc}
\usepackage{amsmath}
\usepackage{amssymb}
\usepackage{graphicx}
\usepackage{amssymb}
\usepackage{mathrsfs}
\usepackage{multirow}
\usepackage{xcolor}
\usepackage[textwidth=2.5cm, textsize=tiny]{todonotes}
\usepackage{cleveref}
\usepackage[noend]{algpseudocode}
\usepackage{graphicx}
\usepackage{subcaption}
\DeclareMathOperator*{\argmax}{arg\!\max}
\DeclareMathOperator*{\argmin}{arg\!\min}
\DeclareMathOperator*{\rmse}{RMSE}
\DeclareMathOperator*{\snr}{SNR}
\DeclareMathOperator*{\mae}{MAE}
\DeclareMathOperator*{\Var}{Var}
\definecolor{mygreen}{rgb}{0,0.40,0}

\title{Linear Support Vector Regression with Linear Constraints\thanks{Submitted.\funding{This work was partly supported by ANR GraVa ANR-18-CE40-0005, Projet ANER RAGA G048CVCRB-2018ZZ and INSERM Plan cancer 18CP134-00.}}}
\author{
    Quentin Klopfenstein\thanks{Institut Mathématique de Bourgogne, Université de Bourgogne, Dijon, France (\email{quentin.klopfenstein@u-bourgogne.fr}).}
    \and
    Samuel Vaiter\thanks{CNRS \& Institut Mathématique de Bourgogne, Université de Bourgogne, Dijon, France (\email{samuel.vaiter@u-bourgogne.fr}).}}
\date{}

\headers{Linear SVR with Linear Constraints}{Q. Klopfenstein and S. Vaiter}

\begin{document}

\maketitle

\begin{abstract}
    This paper studies the addition of linear constraints to the Support Vector Regression (SVR) when the kernel is linear.
    Adding those constraints into the problem allows to add prior knowledge on the estimator obtained, such as finding probability vector or monotone data. 
    We propose a generalization of the Sequential Minimal Optimization (SMO) algorithm for solving the optimization problem with linear constraints and prove its convergence. 
    Then, practical performances of this estimator are shown on simulated and real datasets with different settings: non negative regression, regression onto the simplex for biomedical data and isotonic regression for weather forecast.
\end{abstract}


                              
\section{Introduction}
The Support Vector Machine (SVM) \cite{Vapnik_SVM} is a class of supervised learning algorithms that have been widely used in the past 20 years for classification tasks and regression. 
These algorithms rely on two main ideas: the first one is the maximum margin hyperplane which consists in finding the hyperplane that maximises the distance between the vectors that are to be classified and the hyperplane.
The second idea is the kernel method that allows the SVM to be used to solve non-linear problems. The technic is to map the vectors in a higher dimensional space which is done by using a positive definite kernel, then a maximum margin hyperplane is computed in this space which gives a linear classifier in the high dimensional space. In general, it leads to a non-linear classifier in the original input space.

\paragraph{From SVM to Support Vector Regression}
Different implementations of the algorithms haven been proposed such as C-SVM, $\nu$-SVM \cite{Scholkopf}, Least-Squares SVM \cite{LSsvm}, Linear Programming SVM \cite{lpsvm} among others. Each of these versions have their strenghs and weaknesses depending on which application they are used. They differ in terms of constraints considered for the hyperplane (C-SVM  and Least-Squares SVM), in terms of norm considered on the parameters (C-SVM and Linear Programming SVM) and in terms of optimization problem formulation (C-SVM and $\nu$-SVM). Overall, these algorithms are a great tool for classification tasks and they have been used in many different applications like facial recognition \cite{facial}, image classification \cite{image}, cancer type  classification \cite{cancer}, text categorization \cite{text} to only cite a few examples. Even though, SVM was first developped for classification, an adaptation for regression estimation was proposed in \cite{SVR} under the name Support Vector Regression (SVR).
In this case, the idea of maximum margin hyperplane is slightly changed into finding a tube around the regressors. The size of the tube is controlled by a hyperparameter chosen by the user: $\epsilon$.
This is equivalent to using an $\epsilon$-insensitive loss function, $|y-f(x)|_{\epsilon} = \max\{0, |y-f(x)|-\epsilon\}$ which only penalizes the error above the chosen $\epsilon$ level. 
As for the classification version of the algorithm, a $\nu$-SVR method exists. 
In this version, the hyperparameter $\epsilon$ is computed automatically but a new hyperparameter $\nu$ has to be chosen by the user which controls asymptotically the proportions of support vectors \cite{Scholkopf}. SVR has proven to be a great tool in the field of function estimation for many different applications: predicting times series in stock trades \cite{timeseries}, travel-time prediction \cite{traveltime} and for estimating the amount of cells present inside a tumor \cite{cibersort}.

\paragraph{Incorporating priors}
In this last example of application, the authors used SVR to estimate a vector of proportions, however the classical SVR estimator does not take into account the information known about the space in which the estimator lives. Adding this prior information on the estimator may lead to better estimation performance. Incorporating information in the estimation process is a wide field of studies in statistical learning (we refer to Figure 2 in~\cite{priorSVR} for a quick overview in the context of SVM). 
A growing interest in prior knowledge incorporated as regularization terms has emerged in the last decades. Lasso \cite{lasso}, Ridge \cite{ridge}, elastic-net \cite{elasticnet} regression are examples of regularized problem where a prior information is used to fix an ill-posed problem or an overdetermined problem. The $\ell_{1}$ regularization of the Lasso will force the estimator to be sparse and bring statistical guarantees of the Lasso estimator in high dimensional settings. 
Another commun way to add prior knowledge on the estimator is to add constraints known a-priori on this estimator. The most commun examples are the ones that constrain the estimator to live in a subspace such as Non Negative Least Squares Regression (NNLS) \cite{Lawson}, isotonic regression \cite{isotonic}. These examples belong to a more general type of constraints: linear constraints. Other types of constraints exist like constraints on the derivative of the function that is to be estimated, smoothness of the function for example. Adding those constraints on the Least Squares estimator has been widely studied \cite{isotonic,inequality_LS, NNLS_alg} and similar work has been done for the Lasso estimator in \cite{ConstrainedLasso}. Concerning the SVR, inequality and equality constraints added as prior knowledge were studied in \cite{priorSVR}.
In this paper, the authors described a method for adding linear constraints on the Linear Programming SVR \cite{lpsvm}. This implementation of the algorithm considers the $\ell_{1}$ norm of the parameters in the optimization problem instead of the classical $\ell_{2}$ norm which leads to a linear programming optimization problem to solve instead of a quadratic programming problem. They also described a method for using information about the derivative of the function that is estimated. 

\paragraph{Sequential Minimal Optimization}
One of the main challenges of adding these constraints is that it often increases the difficulty of solving the optimization problem related to the estimator. For example, the Least Squares optimization problem has a closed form solution whereas the NNLS uses sophisticated algorithms \cite{NNLS_alg} to approach the solution. SVM and SVR algorithms were extensively studied and used in practise because very efficient algorithms were developped to solve the underlying optimization problems. 
One of them is called Sequential Minimal Optimization (SMO) \cite{SMO} and is based on a well known optimization technic called coordinate descent. The idea of the coordinate descent is to break the optimization problem into sub-problems selecting one coordinate at each step and minimizing the function only via this chosen coordinate. The developpement of parallel algorithms have increased the interest in these coordinate descent methods which show to be very efficient for large scale problems. One of the key settings for the coordinate descent is the choice of the coordinate at each step, the choice's strategy will affect the efficiency of the algorithm. There exists three families of strategies for coordinate descent: cyclic \cite{Tseng}, random \cite{nesterov} and greedy.  The SMO algorithm is a variant of a greedy coordinate descent \cite{Wright} and is the algorithm implemented in LibSVM \cite{libsvm}. It is very efficient to solve SVM/SVR optimization problems.
In the context of linear kernel, other algorithm are used such as dual coordinate descent~\cite{largescale_SVM} or trust region newton methods~\cite{lin2008trust}.

\paragraph{Priors and SMO}
In one of the application of SVR cited above, information a-priori about the estimator is not used in the estimation process and is only used in a post-processing step. 
This application comes from the cancer research field, where regression algorithms have been used to estimate the proportions of cell populations that are present inside a tumor (see \cite{survey_unmixing} for a survey). 
Several estimators have been proposed in the biostatistics litterature, most of them based on constrained least squares \cite{Abbas, Qiao, Gong} but the gold standard is the estimator based on the Support Vector Regression \cite{cibersort}.
Our work is motivated by incorporating the fact that the estimator for this application belongs to the simplex: $\mathcal{S} = \{x \in\mathbb{R}^{n}: \sum_{i=1}^{n}x_{i}=1,\text{ }x_{i}\geq 0\}$ in the SVR problem. We believe that for this application, it will lead to better estimation performance. 
From an optimization point of view, our motivation is to find an efficient algorithm that is able to solve the SVR optimization problem where generic linear constraints is added to the problem as prior knowledge, including simplex prior as described. This work follows the one from $\cite{priorSVR}$ except that in our case, we keep the $\ell_{2}$ norm on the parameters in the optimization problem which is the most commun version of the SVR optimization problem and we only focus on inequality and equality constraints as prior knowledge.  

\paragraph{Contributions}
In this paper, we study a linear SVR with linear constraints optimization problem. We show that the dual of this new problem shares similar properties with the classical $\nu$-SVR optimization problem (\cref{prop_structure}). We also prove that adding linear constraints to the SVR optimization problem does not change the nature of its dual problem, in the fact that the problem stays a semi-definite positive quadratic function subject to linear constraints.
We propose a generalized SMO algorithm that allows the resolution of the new optimization problem. We show that the updates in the SMO algorithm keep a closed form (\cref{def_updates}) and prove the convergence of the algorithm to a solution of the problem (\cref{convergence_theorem}). We illustrate on synthetic and real datasets the usefulness of our new regression estimator under different regression settings: non-negative regression, simplex regression and isotonic regression.

\paragraph{Outline}
The article proceeds as follows: we introduce the optimization problem coming from the classical SVR and describe the modifications brought by adding linear constraints in \cref{C-SVR}. We then present the SMO algorithm, its generalization for solving constrained SVR and present our result on the convergence of the algorithm in \cref{SMO}. In \cref{expe}, we use synthetic and real datasets on different regression settings to illustrate the practical performance of the new estimator.

\paragraph{Notations}
We write $||.||$ (resp. $\langle .,.\rangle$) for the euclidean norm (resp. inner product) on vectors. We use the notation $X_{:i}$ (resp. $X_{i:}$) to denote the vector corresponding the the $i^{th}$ column of the matrix $X$ (resp. $i^{th}$ row of the matrix $X$). Throughout this paper, the design matrix will be $X\in \mathbb{R}^{n\times p}$ and $y\in\mathbb{R}^{n}$ will be the response vector. $X^{T}$ will be used for the transposed matrix of $X$. The vector \textbf{e} denote the vector with only ones on each of its coordinates and $e_{j}$ denotes the canonical vector with a one at the $j^{\text{th}}$ coordinate. $\nabla_{x_{i}} f$ is the partial derivative $\frac{\partial f}{\partial x_{i}}$.

\section{Constrained Support Vector Regression}
 \label{C-SVR}
First we introduce the optimzation problem related to adding linear constraints to the SVR and discuss some interesting properties about this problem. 
\subsection{Previous work : $\nu$-Support Vector Regression}
The $\nu$-SVR estimator \cite{Scholkopf} is obtained solving the following quadratic optimization problem: 
\begin{equation}\tag{\text{SVR-P}}
    \begin{aligned}
    & \underset{\beta,\beta_{0},\xi_{i},\xi_{i}^{*},\epsilon}{\min}
    &\frac{1}{2}||\beta||^{2}+C(\nu\epsilon + \frac{1}{n}\sum_{i=1}^{n}(\xi_{i}+\xi_{i}^{*})) \\
    & \text{subject to}
    &y_{i}-\beta^{T}X_{i:}-\beta_{0}\leq \epsilon + \xi_{i} \\
    & & \beta^{T}X_{i:}+\beta_{0} -y_{i}\leq \epsilon+\xi_{i}^{*} \\
    & & \xi_{i},\xi_{i}^{*}\geq 0,\epsilon\geq 0.
    \end{aligned}
    \label{svr}
\end{equation}

By solving problem \cref{svr}, we seek a linear function $f(x)=\beta^{T}x+\beta_{0}$ where $\beta\in\mathbb{R}^{p}$ and $\beta_{0}\in\mathbb{R}$, that is at most $\epsilon$ deviating from the response vector coefficient $y_{i}$. This function does not always exist which is why slack variables $\xi \in \mathbb{R}^{n}$ and $\xi^{*}\in\mathbb{R}^{n}$ are introduced in the optimization problem to allow some observations to break the condition given before. $C$ and $\nu$ are two hyperparameters. $C\in \mathbb{R}$ controls the tolerated error and $\nu\in [0,1]$ controls the number of observations that will lay inside the tube of size $2\epsilon$ given by the two first constraints in \cref{svr}. It can be seen as an $\epsilon$-insensitive loss function where a linear penalization is put on the observations that lay outside the tube and the observations that lay inside the tube are not penalized (see \cite{Smola} for more details). \\
\\
The different algorithms proposed to solve \cref{svr} often use its dual problem like in \cite{SMO, largescale_SVM}. The dual problem is also a quadratic optimization problem with linear constraints but its structure allows an efficient resolution as we will see in more details in \cref{SMO}. The dual problem of \cref{svr} is the following optimization problem:

\begin{equation}\tag{\text{SVR-D}}
    \begin{aligned}
        & \underset{\alpha,\alpha^{*}}{\min}
        &\frac{1}{2}(\alpha-\alpha^{*})^{T}Q(\alpha-\alpha^{*})+y^{T}(\alpha-\alpha^{*}) \\
        & \text{subject to}
        &  0\leq \alpha_{i},\alpha_{i}^{*}\leq \frac{C}{n} \\
        & &  \textbf{e}^{T}(\alpha+\alpha^{*})\leq C\nu\\
        & & \textbf{e}^{T}(\alpha-\alpha^{*})=0, \\
    \end{aligned}
    \label{svr_dual}
\end{equation}

where $Q=XX^{T} \in \mathbb{R}^{2n\times 2n}$.

The equation link between \cref{svr} and \cref{svr_dual} is given by the following formula: 
\begin{equation*}
    \begin{aligned}
\beta=-\sum_{i=1}^{n}(\alpha_{i}-\alpha_{i}^{*})X_{i:}.
\end{aligned}
\label{eqlink_svr}
\end{equation*}

\subsection{The constrained optimization problem}
We propose a constrained version of problem \eqref{svr} that allows the addition of prior knowledge on the linear function $f$ that we seek to estimate. The constrained estimator is obtained solving the optimization problem: 

\begin{equation}\tag{\text{LSVR-P}}
    \begin{aligned}
    & \underset{\beta,\beta_{0},\xi_{i},\xi_{i}^{*},\epsilon}{\min}
    &\frac{1}{2}||\beta||^{2}+C(\nu\epsilon + \frac{1}{n}\sum_{i=1}^{n}(\xi_{i}+\xi_{i}^{*})) \\
    & \text{subject to}
    &\beta^{T}X_{i:}+\beta_{0} -y_{i}\leq \epsilon + \xi_{i} \\
    & & y_{i}-\beta^{T}X_{i:}-\beta_{0}\leq \epsilon+\xi_{i}^{*} \\
    & & \xi_{i},\xi_{i}^{*}\geq 0,\epsilon\geq 0 \\
    & & A\beta \leq b \\
    & & \Gamma\beta = d, \\
    \end{aligned}
    \label{polyhedral_svr}
\end{equation}

where $A\in\mathbb{R}^{k_{1}\times p}$, $\Gamma\in\mathbb{R}^{k_{2}\times p}$,  $\beta\in \mathbb{R}^{p}$, $\xi$, $\xi^{*} \in \mathbb{R}^{n}$ and $\beta_{0}$, $\epsilon$, $\in \mathbb{R}$.

The algorithm that we propose in \cref{SMO} also uses the structure of the dual problem of \cref{polyhedral_svr}. The next proposition introduces the dual problem and some of its properties.

\begin{proposition}
    \label{dual_prop}
    If the set $\{ \beta\in\mathbb{R}^{n}, A\beta\leq b, \Gamma\beta\ = d\}$ is not empty then, 
    \begin{enumerate}
        \item Strong duality holds for \cref{polyhedral_svr}.
        \item The dual problem of \cref{polyhedral_svr} is
        \begin{equation}\tag{\text{LSVR-D}}
            \begin{aligned}
            & \underset{\alpha,\alpha^{*},\gamma,\mu}{\min}
            & \frac{1}{2}\bigg[ (\alpha-\alpha^{*})^{T}Q(\alpha-\alpha^{*})+\gamma^{T}AA^{T}\gamma +\mu^{T}\Gamma\Gamma^{T}\mu\\
            & & +2\sum^{n}_{i=1}(\alpha_{i}-\alpha_{i}^{*})\gamma^{T}AX_{i:}-2\sum^{n}_{i=1}(\alpha_{i}-\alpha_{i}^{*})\mu^{T}\Gamma X_{i:}-2\gamma^{T}A\Gamma^{T}\mu\bigg]\\
            & &+y^{T}(\alpha-\alpha^{*})+\gamma^{T}b-\mu^{T}d\\
            & \text{subject to}
            & 0 \leq \alpha_{i}^{(*)}\leq \frac{C}{n} \\
            & & \textbf{e}^{T}(\alpha+\alpha^{*})\leq C\nu \\
            & & \textbf{e}^{T}(\alpha-\alpha^{*})=0 \\
            & & \gamma_{j}\geq 0. \\
            \end{aligned}
            \label{dual_polyhedral_svr}
        \end{equation}
        \item The equation link between primal and dual is $$\beta= -\sum_{i=1}^{n}(\alpha_{i}-\alpha_{i}^{*})X_{i:} - A^{T}\gamma +\Gamma^{T}\mu.$$ 
    \end{enumerate}
\end{proposition}

The proof of the first statement of the proposition is given in the discussion below whereas the proofs for the two other statements are given in the \cref{proof_section}.
We have that $\alpha$, $\alpha^{*}\in \mathbb{R}^{n}$, $\gamma \in \mathbb{R}^{k_{1}}$ is the vector of Lagrange multipliers associated the the inequality constraint $A\beta\leq b$ which explains the non-negative constraints on its coefficients. $\mu\in\mathbb{R}^{k_{2}}$ are the Lagrange multipliers associated to the equality constraint $\Gamma\beta=d$ which also explains that there is no constraints in the dual problem on $\mu$. The objective function $f$ which we will write in the stacked form as: $$f(\theta)=\theta^{T}\bar{Q}\theta+l^{T}\theta,$$
where 
$$\theta=\begin{bmatrix} 
    \alpha \\
    \alpha^{*} \\
    \gamma \\
    \mu 
    \end{bmatrix}, \text{ }l= \begin{bmatrix}
        y \\ -y \\ b \\ -d
    \end{bmatrix} \in \mathbb{R}^{2n+k_{1}+k_{2}}, \text{ }\bar{Q}=\begin{bmatrix}
      Q & -Q & XA^{T} & -X\Gamma^{T} \\
      -Q & Q & -XA^{T} & X\Gamma^{T} \\
      AX^{T} & -AX^{T} & AA^{T} & -A\Gamma^{T} \\
      -\Gamma X^{T} & \Gamma X^{T} & -\Gamma A^{T} & \Gamma\Gamma^{T}

    \end{bmatrix}$$ is a square matrix of size $2n+k_{1}+k_{2}$.

    An important observation is that this objective function is always convex. The matrix $\bar{Q}$ is the product of the matrix $\begin{bmatrix} 
        X \\
        -X \\
        A \\
        -\Gamma
        \end{bmatrix}$ and its transpose matrix. It means that $\bar{Q}$ is a Gramian matrix and it is positive semi-definite which implies that $f$ is convex. The problem \cref{dual_polyhedral_svr} is then a quadratic programming optimization problem which meets Slater's condition if there exists a $\theta$ that belongs to the feasible domain which we will denote by $\mathcal{F}$. If there is such a $\theta$ we have strong duality holding between problem \cref{polyhedral_svr} and \cref{dual_polyhedral_svr}. The only condition we need to have on $A$ and $\Gamma$ is that they define a non-empty polyhedron in order to be able to solve the optimization problem.

    Our second observation on problem \cref{dual_polyhedral_svr} is that the inequality constraints $\textbf{e}^{T}(\alpha+\alpha^{*})\leq C\nu$ is replaced by an equality constraints in the same way that it was suggested in \cite{Chang} for the classical problem \cref{svr_dual}.

    \begin{proposition}\label{prop_structure}

        If $\epsilon>0$, all optimal solutions of \cref{dual_polyhedral_svr} satisfy
        \begin{enumerate}
            \item  $\alpha_{i}\alpha_{i}^{*}=0$, $\forall i$
            \item  $\textbf{e}^{T}(\alpha+\alpha^{*})=C\nu$
        \end{enumerate}
    \end{proposition}
  The proof is given in \cref{proof_structure}.
  This observation will be important for the algorithm that we propose in \cref{SMO}.

\section{Generalized Sequential Minimal Optimization}
\label{SMO}
In this section we propose a generalization of the SMO algorithm \cite{SMO} to solve problem \cref{dual_polyhedral_svr} and present our main result on the convergence of the proposed algorithm to the solution of \cref{dual_polyhedral_svr}. The SMO algorithm is a variant of greedy coordinate descent \cite{Wright} taking into consideration non-separable constraints, which in our case are the two equality constraints. We start by describing the previous algorithm that solve \cref{svr_dual}.
\subsection{Previous work : Sequential Minimal Optimization}
\label{previous_SMO}
In this subsection, we define $f(\alpha,\alpha^{*})=\frac{1}{2}(\alpha-\alpha^{*})^{T}Q(\alpha-\alpha^{*})+y^{T}(\alpha-\alpha^{*})$ and we note $\nabla f\in \mathbb{R}^{2n}$ its gradient. From \cite{Keerthi}, we rewrite the Karush-Kuhn-Tucker (KKT) conditions in the following way: 

\begin{equation}
    \underset{i\in I_{\text{up}}}{\min}\nabla_{\alpha_{i}} f\geq \underset{j\in I_{\text{low}}}{\max}\nabla_{\alpha_{j}} f
    \label{opt_condition}
    \end{equation}

where $$
I_{\text{up}}(\alpha)=\{i \in \{1,\hdots,l\} : \alpha_{i}<\frac{C}{l}\}
$$

$$
I_{\text{low}}(\alpha)=\{i \in \{1,\hdots,l\} : \alpha_{i}>0\}.
$$

The same condition is written for the $\alpha^{*}$ variables replacing $\alpha_{i}$ by $\alpha_{i}^{*}$ above. These conditions leads to an important definition for the rest of this paper. 

\begin{definition}
    We will say that $(i,j)$ is a violating pair of variables if one of these two conditions is satisfied:
    \begin{equation*}
        \begin{aligned}
        i\in I_{\text{up}}(\alpha), j\in I_{\text{low}}(\alpha) \text{ and } \nabla_{\alpha_{i}}f< \nabla_{\alpha_{j}}f \\
        i\in I_{\text{low}}(\alpha), j\in I_{\text{up}}(\alpha) \text{ and } \nabla_{\alpha_{i}}f> \nabla_{\alpha_{j}}f.
    \end{aligned}
    \end{equation*}
\end{definition}

Because the algorithm SMO does not provide in general an exact solution in a finite number of steps there is a need to relax the optimality conditions which gives a new definition.
\begin{definition}
    We will say that $(i,j)$ is a $\tau$-violating pair of variables if one of these two conditions is satisfied:
    \begin{equation*}
        \begin{aligned}
              i\in I_{\text{up}}(\alpha), j\in I_{\text{low}}(\alpha) \text{ and } \nabla_{\alpha_{i}}f< \nabla_{\alpha_{j}}f -\tau\\
        i\in I_{\text{low}}(\alpha), j\in I_{\text{up}}(\alpha) \text{ and } \nabla_{\alpha_{i}}f> \nabla_{\alpha_{j}}f +\tau.
        \end{aligned}
    \end{equation*}
\end{definition}

The SMO algorithm will then choose at each iteration a pair of violating variables in the $\alpha$ block or in the $\alpha^{*}$ block. Once the choice is done, a subproblem of size two is solved, considering that only the two selected variables are to be minimized in problem \cref{svr_dual}.
The outline of the algorithm is presented in \cref{SMO_alg}.

The choice of the violating pair of variables presented in \cite{Keerthi2} was to always work with the most violating pairs of variables, which means the variables that leads to the largest gap compared to the optimality conditions given in \cref{opt_condition}. This choice is what makes a link with greedy coordinate descent, however greedy here is related to the largest gap with the optimality score and is not related to the largest decrease in the objective function.

The resolution of the subproblem of size two has a closed form. The idea is to use the two equality constraints to go from a problem of size two to a problem of size one. Then, the goal is to minimize a quadratic function of one variable under box constraints which is done easily. We will give more details of the resolution of these subproblems in \cref{updates} for our proposed algorithm.

The proof of convergence of SMO algorithm was given in \cite{Keerthi} without convergence rate. The proof relies on showing that the sequence defined by the algorithm $f(\alpha^{k},(\alpha^{*})^{k})$ is a decreasing sequence and that there cannot be the same violating pair of variables infinitely many times. The linear convergence rate was proved later by Schmidt and She \cite{She2017} as well as the identification of the support vectors in finite time.

\begin{algorithm}[t]
    \caption{SMO algorithm}\label{SMO_alg}
    \begin{algorithmic}
    \Require $\tau$>0
    \State Initializing $\alpha^{0}\in\mathbb{R}^{n}$, $(\alpha^{*})^{0}\in\mathbb{R}^{n}$ in $\mathcal{F}$ and set $k=0$

    \While{$\Delta > \tau$}
    \State $i\gets\underset{i\in I_{\text{up}}}{\argmin }\nabla_{\alpha_{i}}f $\hspace{1em} $j\gets\underset{i\in I_{\text{low}}}{\argmax }\nabla_{\alpha_{j}}f $
    
    \State $i^{*}\gets\underset{i\in I_{\text{up}}^{*}}{\argmin }\nabla_{\alpha_{i}^{*}}f$\hspace{1em} $j^{*}\gets\underset{i\in I_{\text{low}}^{*}}{\argmax }\nabla_{\alpha_{j}^{*}}f$
    
    \State $\Delta_{1}\gets\nabla_{\alpha_{j}}f-\nabla_{\alpha_{i}}f$
    \State $\Delta_{2}\gets\nabla_{\alpha_{j}^{*}}f-\nabla_{\alpha_{i}^{*}}f$
    \State $\Delta\gets\max(\Delta_{1},\Delta_{2})$ \Comment{Select the maximal violating pair}

    \State \If{$\Delta=\Delta_{1}$}
    \State $\alpha^{k+1}\gets$ Solution of subproblem for variables $\alpha_{i}$ and $\alpha_{j}$
    \State \Else
    \State $(\alpha^{*})^{k+1}\gets $Solution of subproblem for variables $\alpha_{i^{*}}$ and $\alpha_{j^{*}}$
    \State \EndIf
    \State $k \gets k+1$
    \EndWhile
    \Return $\alpha^{k}, (\alpha^*)^{k}$
    \end{algorithmic}
    \end{algorithm}
    \subsection{Optimality conditions for the constrained SVR}

    In this subsection we define $f$ as the objective function of problem \cref{dual_polyhedral_svr} and $\nabla f\in \mathbb{R}^{2n+k_{1}+k_{2}}$ its gradient. The Lagrangian of optimization problem \cref{dual_polyhedral_svr} is defined by : 

    \begin{multline*}
        L=f-\sum_{i=1}^{n}(\lambda_{i}\alpha_{i}+\lambda_{i}^{*}\alpha_{i}^{*})+\sum_{i=1}^{n}\beta_{i}(\alpha_{i}-\frac{C}{n})+\beta_{i}^{*}(\alpha_{i}^{*}-\frac{C}{n})\\ -\sigma(\sum_{i=1}^{n}(\alpha_{i}+\alpha_{i}^{*})-C\nu)-\delta\sum_{i=1}^{n}(\alpha_{i}-\alpha_{i}^{*})-\sum_{j=1}^{k_{1}}\eta_{j}\gamma_{j}.
    \end{multline*}

    We then give KKT conditions for each block of variables: \\
\\
    \textbf{The $\alpha$ block}\hspace{1em} \begin{equation*}
        \begin{aligned}
            \nabla_{\alpha_{i}}L= \nabla_{\alpha_{i}}f-\lambda_{i}+\beta_{i}-\sigma-\delta=0 \\
            \lambda_{i}\alpha_{i}=0 \\
            \beta_{i}(\alpha_{i}-\frac{C}{n})=0 \\
            \lambda_{i}\geq 0 \\
            \beta_{i} \geq 0
        \end{aligned}
    \end{equation*}

   We will consider different possiblities of value for $\alpha_{i}$.\\
   
   \textbf{Case 1}- $\alpha_{i}=0$ then $\beta_{i}=0$ and $\lambda_{i}\geq 0$\\
   \begin{equation*}
       \nabla_{\alpha_{i}} f-\sigma-\delta\geq 0
   \end{equation*}

   \textbf{Case 2}- $\alpha_{i}=\frac{C}{n}$ then $\lambda_{i}=0$ and $\beta_{i}\geq 0$ \\
   \begin{equation*}
       \nabla_{\alpha_{i}} f-\sigma-\delta\leq 0
   \end{equation*}
   
   \textbf{Case 3}- $0<\alpha_{i}<\frac{C}{n}$ then $\beta_{i}=0$, $\theta_{i}=0$ \\
   \begin{equation*}
       \nabla_{\alpha_{i}} f-\sigma-\delta=0
   \end{equation*}
   
   We then consider the set of indices :
   $$
   I_{\text{up}}(\alpha)=\{i \in \{1,\hdots,n\} : \alpha_{i}<\frac{C}{n}\}
   $$
   
   $$
   I_{\text{low}}(\alpha)=\{i \in \{1,\hdots,n\} : \alpha_{i}>0\}
   $$

   The optimality conditions are satisfied if and only if 
   \begin{equation*}
   \underset{i\in I_{\text{up}}}{\min}\nabla_{\alpha_{i}} f\geq \underset{j\in I_{\text{low}}}{\max}\nabla_{\alpha_{j}} f.
   \end{equation*}
\\
   \textbf{The $\alpha^{*}$ block}\hspace{1em}
In this block, the conditions are very similar to the ones given for the block $\alpha$, the only difference here is that we will have two new sets of indices: 

$$
I_{\text{up}}^{*}(\alpha^{*})=\{i \in \{1,\hdots,n\} : \alpha_{i}^{*}<\frac{C}{n}\}
$$
and
$$
I_{\text{low}}^{*}(\alpha)=\{i \in \{1,\hdots,n\} : \alpha_{i}^{*}>0\}
$$

which gives the following optimality condition: 
\begin{equation*}
\underset{i\in I_{\text{up}}^{*}}{\min}\nabla_{\alpha_{i}^{*}} f\geq \underset{j\in I_{\text{low}}^{*}}{\max}\nabla_{\alpha_{j}^{*}} f.
\end{equation*}
\\
\textbf{The $\gamma$ block}\hspace{1em}

\begin{equation*}
    \nabla_{\gamma_{j}}L= \nabla_{\gamma_{j}}f-\eta_{j}=0 \\
    \eta_{j}\gamma_{j}=0 \\
    \eta_{j}\geq 0
\end{equation*}

We will consider different possiblities of value for $\gamma_{j}$.\\

\textbf{Case 1}- $\gamma_{j}=0$ then \\
\begin{equation*}
    \nabla_{\gamma_{j}} f\geq 0
\end{equation*}

\textbf{Case 2}- $\gamma_{j}>0$ \\
\begin{equation*}
    \nabla_{\gamma_{j}} f= 0
\end{equation*}

\begin{definition}
    We will say that $j$ is a $\tau$-violating variable for the block $\gamma$ if 
    \begin{equation*}
        \nabla_{\gamma_{j}}f + \tau < 0.
    \end{equation*}
\end{definition}

\textbf{The $\mu$ block}\hspace{1em}

\begin{equation*}
    \nabla_{\mu_{j}}L= \nabla_{\mu_{j}}f=0 
\end{equation*}

\begin{definition}
    We will say that $j$ is a $\tau$-violating variable for the block $\mu$ if 
    \begin{equation*}
        |\nabla_{\mu_{j}} f|> \tau.
    \end{equation*}
\end{definition}

From these conditions on each block, we build an optimization strategy that follows the idea of the SMO described in \cref{previous_SMO}. For each block of variables, we compute what we call a \emph{violating optimality score} based on the optimality conditions given above. Once the scores are computed for each block, we select the block which has the largest score and solve an optimization subproblem in the block selected. If the block $\alpha$ or the block $\alpha^{*}$ is selected, we will update a pair of variables by solving a minization problem of size two. However if the block $\gamma$ or the block $\mu$ is selected, we will update only one variable at a time. This is justified by the fact that the variables $\alpha$ and $\alpha^{*}$ have non-separable equality constraints linking them together. The rest of this section will be dedicated to the presentation of our algorithm and to giving some interesting properties such as a closed form for updates on each of the blocks and a convergence theorem.

\begin{algorithm}[t]
    \caption{Generalized SMO algorithm}\label{euclid}
    \begin{algorithmic}
    \Require $\tau>0$
    \State Initializing $\alpha^{0}\in\mathbb{R}^{n}$, $(\alpha^{*})^{0}\in\mathbb{R}^{n}$, $\gamma^{0}\in\mathbb{R}^{k_{1}}$ and $\mu^{0}\in\mathbb{R}^{k_{2}}$ in $\mathcal{F}$ and set $k=0$
    \While{$\Delta > \tau$}
    \State $i\gets\underset{i\in I_{\text{up}}}{\argmin }\nabla_{\alpha_{i}}f $\hspace{5em} $j\gets\underset{i\in I_{\text{low}}}{\argmax }\nabla_{\alpha_{j}}f $
    \State $i^{*}\gets\underset{i\in I_{\text{up}}^{*}}{\argmin }\nabla_{\alpha_{i}^{*}}f$ \hspace{3.7em} $j^{*}\gets\underset{i\in I_{\text{low}}^{*}}{\argmax }\nabla_{\alpha_{j}^{*}}f$
    \State $\Delta_{1}\gets\nabla_{\alpha_{j}}f-\nabla_{\alpha_{i}}f$ \hspace{2.5em}
    $\Delta_{2}\gets\nabla_{\alpha_{j}^{*}}f-\nabla_{\alpha_{i}^{*}}f$
    \State $\Delta_{3}\gets-\underset{j\in \{1,\hdots,k_{1}\}}{\min} \nabla_{\gamma_{j}}f$\hspace{1.5em}
    $\Delta_{4}\gets\underset{j\in \{1,\hdots,k_{2}\}}{\max} |\nabla_{\mu_{j}}f|$
    \State
    \State $\Delta\gets\max(\Delta_{1},\Delta_{2},\Delta_{3},\Delta_{4})$ \Comment{Select the maximal violating variables}
    \State

    \State \If{$\Delta=\Delta_{1}$}
    \State $\alpha^{k+1}\gets$Solution of subproblem for variables $\alpha_{i}$ and $\alpha_{j}$
    \State \ElsIf{$\Delta=\Delta_{2}$}
    \State $(\alpha^{*})^{k+1}\gets$Solution of subproblem for variables $\alpha_{i^{*}}$ and $\alpha_{j^{*}}$
    \State \ElsIf{$\Delta=\Delta_{3}$}
    \State $u=\underset{i\in \{1,\hdots,k_{1}\}}{\argmin} \nabla_{\gamma_{i}}f$
    \State $\gamma^{k+1}\gets$Solution of subproblem for variable $\gamma_{u}$
    \State \Else
    \State $u=\underset{i\in \{1,\hdots,k_{2}\}}{\argmax} \nabla_{\mu_{i}}f$
    \State $\mu^{k+1}\gets$ Solution of subproblem for variable $\mu_{u}$
    \State \EndIf
    \State $k \gets k+1$
    \EndWhile
    \Return $\alpha^{k}, (\alpha^*)^{k}$, $\gamma^{k}$, $\mu^{k}$

    \end{algorithmic}
    \end{algorithm}

    \subsection{Updates rules and convergence}
    \label{updates}
    The first definition describes the closed form updates for the different blocks of variables.

    \begin{definition}
        \label{def_updates}
        The update between iterate $k$ and iterate $k+1$ of the generalized SMO algorithm has the following form:
        \begin{enumerate}
            \item if the block $\alpha$ is selected and $(i,j)$ is the most violating pair of variable then the update will be as follows:
            \begin{equation*}
                \begin{aligned}
                    \alpha_{i}^{k+1}=\alpha_{i}^{k}+t^{*} \\
                    \alpha_{j}^{k+1}=\alpha_{j}^{k}-t^{*}, 
                \end{aligned}
            \end{equation*}
           where $t^{*}=\min(\max(I_{1},-\frac{(\nabla_{\alpha_{i}}f-\nabla_{\alpha_{j}}f)}{(Q_{ii}-2Q_{ij}+Q_{jj})}),I_{2})$ with $I_{1}=\max(-\alpha_{i}^{k},\alpha_{j}^{k}-\frac{C}{n})$ and $I_{2}=\min(\alpha_{j}^{k},\frac{C}{n}-\alpha_{i}^{k})$.
            
           \item if the block $\alpha^{*}$ is selected and $(i^{*},j^{*})$ is the most violating pair of variable then the update will be as follows:
           \begin{equation*}
            \begin{aligned}
                (\alpha_{i}^{*})^{k+1}=(\alpha_{i}^{*})^{k}+t^{*} \\
               (\alpha_{j}^{*})^{k+1}=(\alpha_{j}^{*})^{k}-t^{*}, 
            \end{aligned}
           \end{equation*}
          where $t^{*}=\min(\max(I_{1},-\frac{(\nabla_{\alpha_{i}^{*}}f-\nabla_{\alpha_{j}^{*}}f)}{(Q_{ii}-2Q_{ij}+Q_{jj})}),I_{2})$ with $I_{1}=\max(-(\alpha_{i}^{*})^{k},(\alpha_{j}^{*})^{k}-\frac{C}{n})$ and $I_{2}=\min((\alpha_{j}^{*})^{k},\frac{C}{n}-(\alpha_{i}^{*}))^{k}$.
           \item if the block $\gamma$ is selected and $i$ is the index of the most violating variable in this block then the update will be as follows:
           $$
           \gamma_{i}^{k+1}=\max(-\frac{\nabla_{\gamma_{i}}f}{(AA^{T})_{ii}}+\gamma_{i}^{k},0).
           $$
           \item if the block $\mu$ is selected and $i$ is the index of the most violating variable in this block then the update will be as follows:
           $$
           \mu_{i}^{k+1}=-\frac{\nabla_{\mu_{i} f}}{(\Gamma\Gamma^{T})_{ii}}+\mu_{i}^{k}.
           $$
        \end{enumerate}
    \end{definition}

    This choice of updates comes from solving the optimization problem \cref{dual_polyhedral_svr} considering that only one or two variables are updated at each step. One of the key elements of the algorithm is to make sure that at each step the iterate belongs to $\mathcal{F}$. Let's suppose that the block $\alpha$ is selected as the block in which the update will happen and let $(i,j)$ be the most violating pair of variables. The update is the resolution of a subproblem of size 2, considering that only $\alpha_{i}$ and $\alpha_{j}$ are the variables, the rest remains constant. The two equality constraints in \cref{dual_polyhedral_svr}, $\sum_{i=1}^{n}\alpha_{i}-\alpha_{i}^{*}=0$ and $\sum_{i=1}^{n}\alpha_{i}+\alpha_{i}^{*}=C\nu$, lead to the two following equalities: $\alpha_{i}^{k+1}+\alpha_{j}^{k+1}=\alpha_{i}^{k}+\alpha_{j}^{k}$.  The later yields to using a parameter $t$ for the update of the variables leading to: 
    \begin{equation*}
        \begin{aligned}
            \alpha_{i}^{k+1}=\alpha_{i}^{k}+t, \\
            \alpha_{j}^{k+1}=\alpha_{j}^{k}-t.
        \end{aligned}
    \end{equation*}
    Updating the variable in the block $\alpha$ this way will force the iterates of \cref{SMO_alg} to meet the two equalities constraints at each step.
    We find $t$ by solving \cref{dual_polyhedral_svr} considering that we minimize only over $t$. Let $u\in\mathbb{R}^{2n+p+k_{1}+k_{2}}$ be the vector that contains only zeros except at the $i^{th}$ coordinate where it is equal to $t$ and at $j^{th}$ coordinate where it is equal to $-t$. 
    Therefore, we find $t$ by minimizing the following optimization problem:
    \begin{equation*}
        \begin{aligned}
            & \underset{t \in \mathbb{R}}{\min}
            & \psi (t)=\frac{1}{2}\bigg[ (\theta^{k}+u)^{T}\bar{Q}(\theta^{k}+u)\bigg]+l^{T}(\theta^{k}+u)\\
            & \text{subject to}
            & 0 \leq \alpha_{i}^{k+1},\alpha_{j}^{k+1}\leq \frac{C}{n}. \\
            \end{aligned}
    \end{equation*}

    First we minimize the objective function without the constraints and since it is a quadratic function of one variable we just clip the solution of unconstrained problem to have the solution of the constrained problem. We will use the term "clipped update" or "clipping" when the update is projected unto the constraints space and is not the result of the unconstrainted optimization problem. As we only consider size one problem for the updates, it will mean that the update will be a bound of an interval. We will use the notation K as a term containing the terms that do not depend on $t$. We write that 

    \begin{align*}
        \psi (t) &=\frac{1}{2}u^{T}\bar{Q}u +u^{T}\bar{Q}\theta^{k}+l^{T}u+ K \\
                 & = \frac{1}{2}t^{2}(\bar{Q}_{ii}+\bar{Q}_{jj}-2\bar{Q}_{ij})+u^{T}\nabla f(\theta^{k}) + K \\
                 & = \frac{1}{2}t^{2}(\bar{Q}_{ii}+\bar{Q}_{jj}-2\bar{Q}_{ij})+t(\nabla_{\alpha_{i}} f(\theta^{k})-\nabla_{\alpha_{j}} f(\theta^{k})) + K.
    \end{align*}
    It follows that the unconstrained minimum of $\psi (t)$ is $t_{q}=\frac{-(\nabla_{\alpha_{i}} f(\theta^{k})-\nabla_{\alpha_{j}} f(\theta^{k}))}{(\bar{Q}_{ii}+\bar{Q}_{jj}-2\bar{Q}_{ij})}$. Taking the constraints into account we have that:
    \begin{align*}
        0\leq \alpha_{i}^{k}+t\leq \frac{C}{n}, \\
        0\leq \alpha_{j}^{k}-t\leq \frac{C}{n},
    \end{align*}
    it yields to $t^{*}=\min(\max(I_{1},t_q),I_{2})$ with $I_{1}=\max(-\alpha_{i},\alpha_{j}-\frac{C}{n})$ and $I_{2}=\min(\alpha_{j},\frac{C}{n}-\alpha_{i})$.
    The definition of the updates for the block $\alpha^{*}$ relies on the same discussion.

    Let's now make an observation that will explain the definition of the updates for the blocks $\gamma$ and $\mu$. Let $i$ be the index of the variable that will be updated. Solving the problem:
    \begin{equation*}
         \theta_{i}^{k+1}=\underset{\theta_{i}}{\argmin}\frac{1}{2}\theta^{T}\bar{Q}\theta+l^{T}\theta,
    \end{equation*}
    leads to the following solution $\theta^{k+1}_{i}=\frac{-\nabla_{i}f(\theta^{k})}{\bar{Q}_{ii}}+\theta_{i}^{k}$. 

    Let's recall that the update for the block $\gamma$ has to keep the coefficient of $\gamma$ positive to stay in $\mathcal{F}$ hence we have to perform the following clipped update with $i\in\{2n+p+1,\hdots,2n+p+k_{1}\}$:
    $$
    \theta_{i}^{k+1}=\max(\frac{-\nabla_{\gamma_{i}}f(\theta^{k})}{\bar{Q}_{ii}}+\theta^{k}_{i},0).
    $$
    Then noticing that $\bar{Q}_{ii} = AA^{T}_{ii}$ for this block, we obtain the update for the block $\gamma$. 

    There are no constraints on the variables in the blok $\mu$, so the update comes from the fact that $\bar{Q}_{ii}= \Gamma\Gamma^{T}_{ii}$  for $i\in \{ 2n+p+k_{1}+1,\hdots,2n+p+k_{1}+k_{2} \}$ which corresponds to the indices of the block $\mu$.

 From these updates we have to make sure that $Q_{ii}+Q_{jj}-2Q_{ij}\neq 0$, let us recall that $Q_{ij}=\langle X_{i:}, X_{j:}\rangle$ which means that $Q_{ii}+Q_{jj}-2Q_{ij}= ||X_{i:}-X_{j:}||^{2}$. This quantity is zero only when $X_{i:}=X_{j:}$ coordinate wise. It would mean that the same row appears two times in the design matrix which does not bring any new information for the regression and can be avoided easily. $(AA^{T})_{ii}=\langle A_{i:}, A_{i:}\rangle$ is zero if and only if $A_{i:}=0$ which means that a row of the matrix A is zero, so there is no constraint on any variable of the optimization problem which will never happen. It is the same discussion for $(\Gamma\Gamma^{T})_{ii}$.

    The next proposition makes sure that once a variable (resp. pair of variables) is updated, it cannot be a violating variable (resp. pair of variables) at the next step. This proposition makes sure, for the two blocks $\alpha$ and $\alpha^{*}$, that the update $t^{*}$ cannot be $0$. 
    \begin{proposition}
        \label{prop_non_violating}
        If $(i,j)$ (resp.$i$) was the pair of most violating variable (resp. the most violating variable) in the block $\alpha$ or $\alpha^{*}$  (resp. block $\gamma$ or $\mu$) at iteration $k$ then at iteration $k+1$, $(i,j)$ (resp. $i$) cannot be violating the optimality conditions.  
    \end{proposition}
    The proof of this proposition is left in the \cref{proof_non_violating}. 

    Finally, we show that the algorithm converges to a solution of \cref{dual_polyhedral_svr} and since strong duality holds it allows us to have a solution of \cref{polyhedral_svr}. 
    \begin{theorem}
        \label{convergence_theorem}
        For any given $\tau>0$ the sequence of iterates $\{\theta^{k}\}$, defined by the generalized SMO algorithm, converges to an optimal solution of the optimization problem \cref{dual_polyhedral_svr}
    \end{theorem}
    
    The proof of this theorem relies on the same idea as the one proposed in \cite{Lopez} for the classical SMO algorithm and is given in \cref{proof_theorem}. We show that it can be extended to our algorithm with some new observations. The general idea of the proof is to see that the distance between the primal vector generated by the SMO-algorithm and the optimal solution of the primal is controled by the following expression $\frac{1}{2}||\beta^{k}-\beta^{\text{opt}}||\leq f(\theta^{k})-f(\theta^{\text{opt}})$, where $\beta^{k}$ is the $k^{th}$ primal iterate obtained via the relationship primal-dual and $\theta^{k}$ and where $\beta^{\text{opt}}$ is a solution of \cref{polyhedral_svr}. From this observation, we show that we can find a subsequence of the SMO-algorithm $\theta^{k_{j}}$ that converges to some $\bar{\theta}$, solution of the dual problem. Using the continuity of the objective function of the dual problem, we have that $f(\theta^{k_{j}})\rightarrow f(\bar{\theta})$. Finally, we show that the sequence $\{f(\theta^{k})\}$ is decreasing and bounded which implies its convergence and from the convergence monotone theorem we know that  to $f(\theta^{k})$ converges to $f(\bar{\theta})$ since one of its subsequence converges. This proves that $||\beta^{k}-\beta^{\text{opt}}||\rightarrow 0$ and finishes the proof. The convergence rate for the SMO algorithm is difficult to obtain considering the greedy choice of the blocks and the greedy choice inside the blocks. A proof for the classical SMO exists but with uniformly at random choice of the block \cite{She2017}. Convergence rate for greedy algorithms in optimization can be found in \cite{Nutini} for example but the assumption that the constraints must be separable is a major issue for our case. The study of this convergence rate is out of scope of this paper. 

\section{Numerical experiments}
\label{expe}

The code for the different regression settings is available on a GitHub repository\footnote{\url{https://github.com/Klopfe/LSVR}}, each setting is wrapped up in a package and is fully compatible with scikit learn \cite{scikit-learn} \texttt{BaseEstimator} class.

In order to compare the estimators, we worked with the Mean Absolute Error (MAE) and the Root Mean Squared Error (RMSE) which are given by the following expressions: 

$$
\mae = \frac{1}{p}\sum_{i=1}^{p} |\beta^{*}_{i}-\hat{\beta}_{i}|,
$$

$$
\rmse=\sqrt{\frac{1}{p}||\beta^{*}-\hat{\beta}||^{2}},
$$
where $\beta^{*}$ are the ground truth coefficients and $\hat{\beta}$ are the estimated coefficients. We also used the Signal-To-Noise Ratio (SNR) to control the level noise simulated in the data. We used the following definition: 
$$
\snr=10\log 10(\frac{\mathbb{E}(X\beta(X\beta)^{T}             )}{\Var(\epsilon)}).
$$

\subsection{Non Negative regression}
First, the constraints are set to force the coefficient of $\beta$ to be positive and we compare our constrained-SVR estimator with the NNLS \cite{Lawson} estimator which is the result of the following optimization problem: 
\begin{equation}\tag{\text{NNLS}}
    \begin{aligned}
    & \underset{\beta}{\min}
    &\frac{1}{2}||y-X\beta|| \\
    & \text{subject to}
    &\beta_{i}\geq 0. \\
\end{aligned}
    \label{NNLS} 
\end{equation}
In this special case of non-negative regression, $A=-I_{p}$, $b=0$, $C=0$, $d=0$, the constrained-SVR optimization problem which we will call Non-Negative SVR (NNSVR) then becomes: 
\begin{equation}\tag{\text{NNSVR}}
    \begin{aligned}
    & \underset{\beta,\beta_{0},\xi_{i},\xi_{i}^{*},\epsilon}{\min}
    &\frac{1}{2}||\beta||^{2}+C(\nu\epsilon + \frac{1}{n}\sum_{i=1}^{n}(\xi_{i}+\xi_{i}^{*})) \\
    & \text{subject to}
    &\beta^{T}X_{i:}+\beta_{0} -y_{i}\leq \epsilon + \xi_{i} \\
    & & y_{i}-\beta^{T}X_{i:}-\beta_{0}\leq \epsilon+\xi_{i}^{*} \\
    & & \xi_{i},\xi_{i}^{*}\geq 0,\epsilon\geq 0 \\
    & & \beta_{i} \geq 0. \\
    \end{aligned}
    \label{nn-lssvr} 
\end{equation}

\paragraph{Synthetic data}
We generated the design matrix $X$ from a gaussian distribution $\mathcal{N}(0,1)$ with 500 samples and 50 features. The true coefficients to be found $\beta^{*}$ were gererated taking the exponential of a gaussian distribution $\mathcal{N}(0,2)$ in order to have positive coefficients. $Y$ was simply computed as the product between $X$ and $\beta^{*}$. We wanted to test the robustness of our estimator compared to NNLS and variant of SVR estimators. To do so, we simulated noise in the data using different types of distributions, we tested gaussian noise and laplacian noise under different levels of noise. For this experiment, the noise distributions were generated to have a SNR equals to 10 and 20, for each type of noise we performed 50 repetitions. The noise was only added in the matrix $Y$ the design matrix $X$ was left noiseless. We compared different estimators NNLS, NNSVR, the Projected-SVR (P-SVR) which is simply the projection of the classical SVR estimator unto the positive orthant and also the classical SVR estimator without constraints. The results of this experiment are in \cref{table1}. We see that for a low gaussian noise level ($\snr = 20$) the NNLS has a lower RMSE and lower MAE. However, we see that the differences between the four compared methods are small. When the level of noise increases ($\snr = 10$), the NNSVR estimator is the one with the lowest RMSE and MAE. The NNLS estimator performs poorly in the presence of high level of noise in comparison to the SVR based estimator. When a laplacian noise is added to the data, the NNSVR is the estimator that has the lowest RMSE and MAE for low level of noise $\snr = 20$ and high level of noise $\snr = 10$. 

\begin{table}
\caption{Results for the Support Vector Regression (SVR), Projected Support Regression (P-SVR), Non-Negative Support Vector Regression (NNSVR) and Non-Negative Least Squares (NNLS) for simulated data with $n= 500$ and $p=50$. The mean (standard deviation) of the Root Mean Squared Error (RMSE) and the Mean Absolute Error (MAE) over 50 repetitions are reported. Different noise distribution (gaussian and laplacian) and different Signal to Noise Ratio (SNR) values were tested.  }
\label{table1}
\centering
\begin{tabular}{|c|c|c|c|}
    \hline
    \textbf{Distribution} & \textbf{Estimator} & \textbf{RMSE} & \textbf{MAE} \\
    \hline
     \multirow{2}{*}{Gaussian noise } & SVR & 2.238 (0.081)& 29.288 (2.452) \\
   
    & P-SVR & 2.178 (0.087) & 27.248 (2.545)\\
    
    SNR = 20 & NNSVR & 2.174 (0.089)& 27.224 (2.480)\\
    ($\sigma = 773.1$) & NNLS & \textbf{2.120 (0.114)} & \textbf{25.226 (2.699)}\\
    
    \hline
     \multirow{2}{*}{Gaussian noise } & SVR & 2.732 (0.099)& 44.764 (4.230) \\
   
    & P-SVR & 2.584 (0.154) & 39.687 (5.963)\\
    
    SNR = 10 & NNSVR & \textbf{2.536 (0.105)}& \textbf{37.740 (3.866)}\\
    ($\sigma = 2444.9$) & NNLS & 3.478 (0.208) & 60.553 (7.923)\\
    \hline
    \multirow{2}{*}{Laplacian noise} & SVR & 2.086 (0.109) & 25.538 (3.181) \\
   
    & P-SVR & 2.039 (0.109) & 23.978 (3.059) \\
    
    SNR = 20& NNSVR & \textbf{2.035 (0.115)} & \textbf{23.827 (3.146)} \\
    (b = 546.7) & NNLS & 2.115 (0.103) & 25.028 (2.571) \\
    \hline
    \multirow{2}{*}{Laplacian noise} & SVR & 2.665 (0.148) & 42.245 (5.777) \\
   
    & P-SVR & 2.526 (0.198) & 37.745 (7.271) \\
    
    SNR = 10& NNSVR & \textbf{2.480 (0.157)} & \textbf{35.786 (5.761)} \\
    (b = 1728.8) & NNLS & 3.463 (0.230) & 63.940 (8.375) \\
    \hline
 \end{tabular}
\end{table}

\subsection{Regression unto the simplex}
\label{simplex_simu}
 In this subsection, we study the performance of our proposed estimator on simplex constraints Simplex Support Vector Regression (SSVR). In this case, $A=-I_{p}$, $b = 0$, $\Gamma = \textbf{e}$ and $d=1$. The optimization problem that we seek to solve is: 

\begin{equation}\tag{\text{SSVR}}
    \begin{aligned}
    & \underset{\beta,\beta_{0},\xi_{i},\xi_{i}^{*},\epsilon}{\min}
    &\frac{1}{2}||\beta||^{2}+C(\nu\epsilon + \frac{1}{n}\sum_{i=1}^{n}(\xi_{i}+\xi_{i}^{*})) \\
    & \text{subject to}
    &\beta^{T}X_{i:}+\beta_{0} -y_{i}\leq \epsilon + \xi_{i} \\
    & & y_{i}-\beta^{T}X_{i:}-\beta_{0}\leq \epsilon+\xi_{i}^{*} \\
    & & \xi_{i},\xi_{i}^{*}\geq 0,\epsilon\geq 0 \\
    & & \beta_{i} \geq 0 \\
    & & \sum_{i}\beta_{i} = 1. \\
    \end{aligned}
    \label{lssvr}
\end{equation}

\paragraph{Synthetic data}
 We first tested on simulated data generated by the function \texttt{make\_regression} of scikit-learn. Once the design matrix $X$ and the response vector $y$ were generated using this function, we had access to the ground truth that we will write $\beta^{*}$. This function was not designed to generate data with a $\beta^{*}$ that belongs to the simplex so we first projected $\beta^{*}$ unto the simplex and then recomputed $y$ multiplying the design matrix by the new projected $\beta^{*}_{\mathcal{S}}$. We added a centered gaussian noise   in the data with the standard deviation of the gaussian was chosen such as the signal-to-noise ratio (SNR) was equal to a defined number, we used the following formula for a given SNR: 
 $$
 \sigma = \sqrt{\frac{\Var(y)}{10^{SNR/10}}},
 $$
 where $\sigma$ is the standard deviation used to simulate the noise in the data.
The choice of the two hyperparameters $C$ and $\nu$ was done using 5-folds cross validation on a grid of possible pairs. The values of $C$ were taken evenly spaced in the $log10$ base between $[-3,3]$, we considered 10 different values. The values of $\nu$ were taken evenly spaced in the linear space between $[0.05,1.0]$ and we also considered 10 possible values. We tested different size for the matrix $X\in \mathbb{R}^{n\times p}$ to check the potential effects of the dimensions on the quality of the estimation and we did 50 repetitions for each point of the curves. The measure that was used to compare the different estimators is the RMSE between the true $\beta$ and the estimated $\hat{\beta}$.

We compared the RMSE of our estimator to the Simplex Ordinary Least Squares (SOLS) which is the result of the following optimization problem: 
\begin{equation}\tag{\text{SOLS}}
    \begin{aligned}
    & \underset{\beta}{\min}
    &\frac{1}{2}||y-X\beta|| \\
    & \text{subject to}
    &\beta_{i}\geq 0, \\
    & &\sum_{i=1}^{p} \beta_{i} = 1, \\
\end{aligned}
    \label{SOLS} 
\end{equation}
and to the estimator proposed in the biostatics litterature that is called Cibersort. This estimator is simply the result of using the classical SVR and project the obtained estimator unto the simplex. The RMSE curves as a function of the SNR are presented in \cref{simplex_simul}. We observe that the SSVR is generally the estimator with the lowest RMSE, this observation becomes clearer as the level of noise increases in the data. We notice that when there is a low level of noise and when $n$ is not too large in comparison to $p$, the three compared estimator perform equally. However, there is a setting when $n$ is large in comparison to $p$ (in this experiment for $n = 250$ or $500$ and $p = 5$) where the SSVR estimator has a higher RMSE than the Cibersort and SOLS estimator untill a certain level of noise ($\snr<15$). Overall, this simulation shows that there is a significant improvement in the estimation performance of the SSVR mainly when there is noise in the data. 

\begin{figure}
    \label{simplex_simul}
    \centering
    \includegraphics[scale=0.85 ]{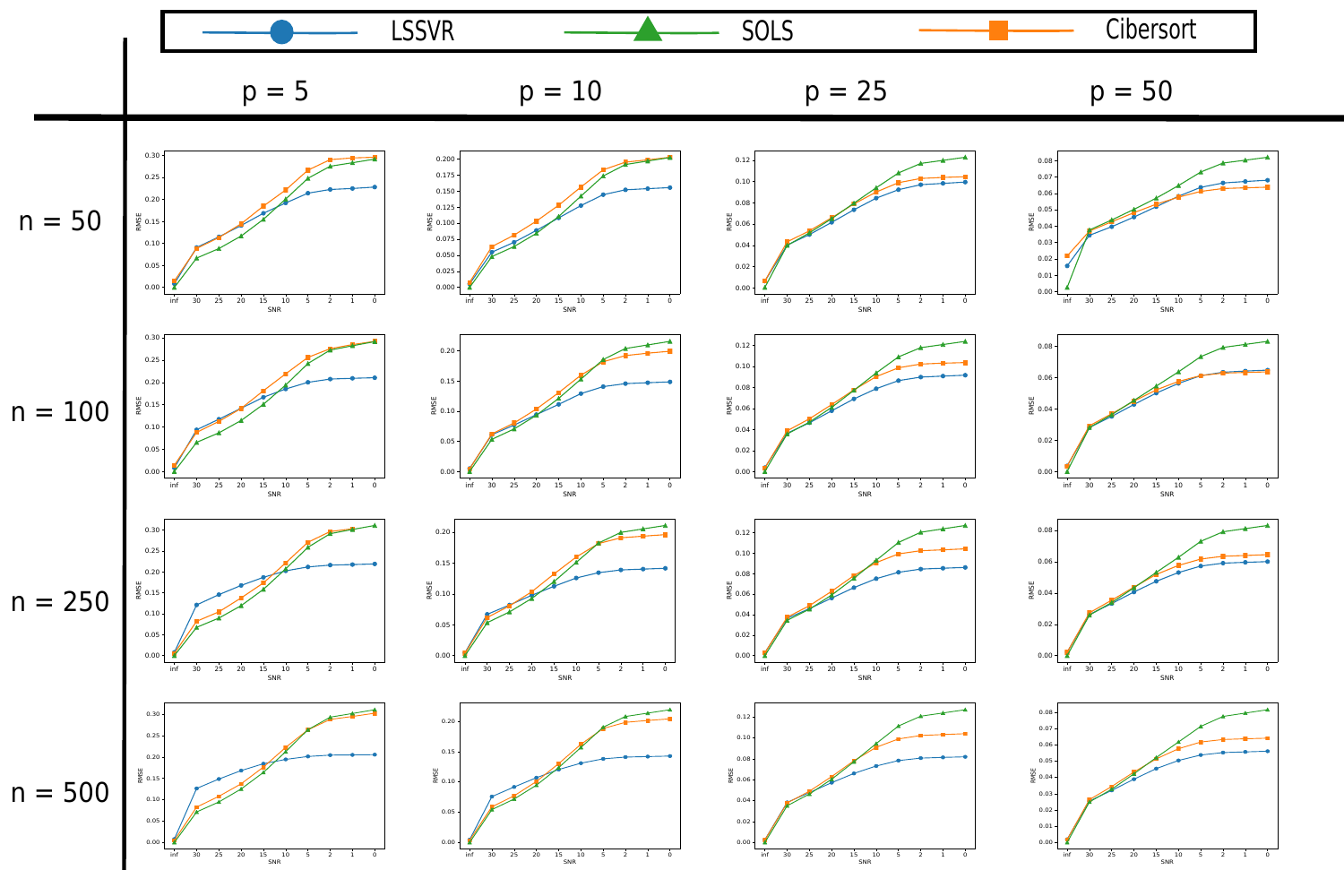}
    \caption{The Root Mean Squared Error (RMSE) as a function of the Signal to Noise Ration (SNR) is presented. Different dimensions for the design matrix $X$ and the response vector $y$ were considered. $n$ represents the number of rows of $X$ and $p$ the number of columns. For each plot, the blue line represents the RMSE for the Linear Simplex SVR (LSSVR) estimator, the green one the Simplex Ordinary Least Squares (SOLS) estimator and the orange on the Cibersort estimator. Each point of the curve is the mean RMSE of 50 repetitions. The noise in the data has a gaussian distribution. }
\end{figure}

\paragraph{Real dataset}
In the cancer research field, regression algorithms have been used to estimate the proportions of cell populations that are present inside a tumor. Indeed, a tumor is composed of different types of cells such as cancer cells, immune cells, healthy cells among others. Having access to the information of the proportions of these cells could be a key to understanding the interactions between the cells and the cancer treatment called immunotherapy \cite{immunotherapy}. The modelization done is that the RNA extracted from the tumor is seen as a mixed signal composed of different pure signals coming from the different types of cells. This signal can be unmixed knowing the different pure RNA signal of the different types of cells. In other words, $y$ will be the RNA signal coming from a tumor and $X$ will be the design matrix composed of the RNA signal from the isolated cells. The number of rows represent the number of genes that we have access to and the number of columns of $X$ is the number of cell populations that we would like to quantify. The hypothesis is that there is a linear relationship between $X$ and $y$. As said above, we want to estimate proportions which means that the estimator has to belong to the probability simplex $\mathscr{S}=\{x : x_{i}\geq 0\text{ , }\sum_{i}x_{i}=1\}$. \\
\\
Several estimators have been proposed in the biostatistics litterature most of them based on constrained least squares \cite{Qiao,Gong,Abbas} but the gold standard is the estimator based on the SVR.

We compared the three same estimators on a real biological dataset where the real quantities of cells to obtain were known. The dataset can be found on the GEO website under the accession code GSE11103\footnote{The dataset can be downloaded from the \href{https://www.ncbi.nlm.nih.gov/geo/}{Gene Expression Omnibus} website under the accession code GSE11103.}. For this example $n=584$ and $p=4$ and we have access to 12 different samples that are our repetitions. Following the same idea than previous benchmark performed in this field of application, we increased the level of noise in the data and compared the RMSE of the different estimators. gaussian and laplacian distributions of noise were added to the data. The choice of the two hyperparameters $C$ and $\nu$ was done using 5-folds cross validation on a grid of possible pairs. The values of $C$ were taken evenly spaced in the $\log_{10}$ base between $[-5,-3]$, we considered 10 different values. The interval of $C$ is different than the simulated data because of the difference in the range value of the dataset. The values of $\nu$ were taken evenly spaced in the linear space between $[0.05,1.0]$ and we also considered 10 possible values. 

We see that when there is no noise in the data ($\snr = \infty$) both Cibersort and SSVR estimator perform equally. The SOLS estimator already has a higher RMSE than the two others estimator probably due to the noise already present in the data. As the level of noise increases, the SSVR estimator remains the estimator with the lowest RMSE in both gaussian and laplacian noise settings.  

\begin{figure}
    \centering
    \includegraphics[scale=0.80 ]{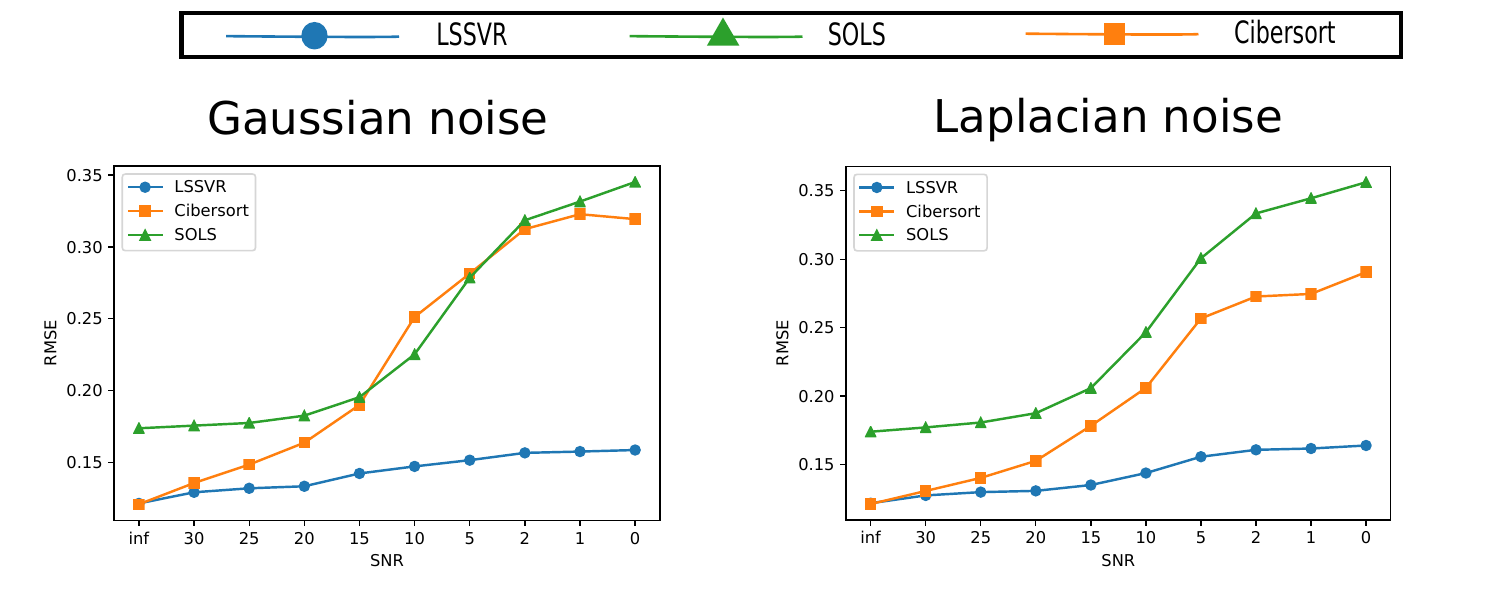}
    \caption{The Root Mean Squared Error (RMSE) as a function of the Signal to Noise Ration (SNR) is presented on a real dataset where noise was manually added. Two different noise distribution were tested: gaussian and laplacian. Each point of the curve is the mean RMSE of 12 different response vectors and we repeated the process four times for each level of noise. This would be equivalent to having 48 different repetitions. }
\end{figure}

\subsection{Isotonic regression}

In this subsection, we will consider constraints that impose an order on the variables. This type of regression is usually called isotonic regression. Such constraints appear when prior knowledge are known on a certain order on the variables. This partial order on the variables can also be seen as an acyclic directed graph. More formally, we note $G = (V, E)$ a directed acyclic graph where $V$ is the set of vertices and $E$ is the set of nodes. On this graph, we define a partial order on the vertices. We will say for $u, v\in V$ that $u\leq v$ if and only if there is a path joining $u$ and $v$ in $G$. This type of constraints seems natural in different applications such as biology, medicine, weather forecast.

The most simple example of this type of constraints might be the monotonic regression where we force the variables to be in a increasing or decreasing order. It means that with our former notations that we would impose that $\beta_{1}\leq \beta_{2} \leq \hdots \leq \beta_{p}$ on the estimator.
This type of constraints can be coded in a finite difference matrix (or more generally any incidence matrix of a graph)
\begin{equation*}
    A =\begin{bmatrix} 
        1 & -1 & 0 & \hdots & 0 \\
        0 & 1 & -1 & \ddots & \vdots \\
        \vdots  & \ddots & \ddots & \ddots & 0\\
        0 & \hdots & 0 & 1 & -1
        \end{bmatrix}
\end{equation*}
and $\Gamma = 0$, $b=0$, $d=0$ forming linear constraints as in the scope of this paper. The Isotonic Support Vector Regression (ISVR) optimization problem is written as follows:

\begin{equation}\tag{\text{ISVR}}
    \begin{aligned}
    & \underset{\beta,\beta_{0},\xi_{i},\xi_{i}^{*},\epsilon}{\min}
    &\frac{1}{2}||\beta||^{2}+C(\nu\epsilon + \frac{1}{n}\sum_{i=1}^{n}(\xi_{i}+\xi_{i}^{*})) \\
    & \text{subject to}
    &\beta^{T}X_{i:}+\beta_{0} -y_{i}\leq \epsilon + \xi_{i} \\
    & & y_{i}-\beta^{T}X_{i:}-\beta_{0}\leq \epsilon+\xi_{i}^{*} \\
    & & \xi_{i},\xi_{i}^{*}\geq 0,\epsilon\geq 0 \\
    & & \beta_{1} \leq \beta_{2}\leq \hdots \leq \beta_{n}. \\
    \end{aligned}
    \label{isvr}
\end{equation}

We compare our proposed ISVR estimator with the classical least squares isotonic regression (IR) \cite{isotonic} which is the solution of the following problem:

\begin{equation}\tag{\text{IR}}
    \begin{aligned}
    & \underset{\beta}{\min}
    &\frac{1}{2}||\beta-y||^{2} \\
    & \text{subject to} & \beta_{1} \leq \beta_{2}\leq \hdots \leq \beta_{n}. \\
    \end{aligned}
    \label{isotonic_ls}
\end{equation}

\paragraph{Synthetic dataset}
We first generated data from a gaussian distribution ($\mu=0$, $\sigma=1$) that we sorted and then added noise in the data following the same process as described in \cref{simplex_simu} with different SNR values (10 and 20). We tested gaussian noise and laplacian noise. 
We compared the estimation quality of both methods using MAE and RMSE. In this experiment, the design matrix $X$ is the identity matrix. We performed grid search selection via cross validation for the hyperparameters $C$ and $\nu$. 
$C$ had 5 different possible values taken on the logscale from 0 to 3, and $\nu$ had 5 different values taken between 0.05 and 1 on the linear scale.
The dimension of the generated gaussian vector was 50 and we did 50 repetitions. 
We present in \cref{table2} the results of the experiment, the value inside a cell is the mean RMSE or MAE over the 50 repetitions and the value between brackets is the standard deviation over the repetitions. Under a low level of gaussian noise or laplacian noise, both methods are close in term of RMSE and MAE with a little advantage for the classical isotonic regression estimator. When the level of noise is important ($\snr = 10$), our proposed ISVR has the lowest RMSE and MAE for the two noise distribution tested.

\begin{table}[t]
    \caption{Results for the Isotonic Support Vector Regression (ISVR), and the Isotonic regression (IR) for simulated data with $p=50$. The mean (standard deviation) of the Root Mean Squared Error (RMSE) and the Mean Absolute Error (MAE) over 50 repetitions are reported. Different noise distribution (gaussian and laplacian) and different Signal to Noise Ratio (SNR) values were tested.}
    \label{table2}
    \centering
    \begin{tabular}{|c|c|c|c|}
        \hline
        \textbf{Distribution} & \textbf{Estimator} & \textbf{RMSE} & \textbf{MAE} \\
        \hline
         Gaussian noise & ISVR & 0.212 (0.02) & 0.254 (0.06)  \\
       
         SNR = 20 & IR & \textbf{0.203 (0.02)} & \textbf{0.229 (0.04)}\\

        \hline
         Gaussian noise & ISVR & \textbf{0.284 (0.04)} & \textbf{0.446 (0.12)}  \\
        SNR = 10 & IR & 0.311 (0.04) & 0.534 (0.12) \\
        \hline
        Laplacian noise & ISVR & \textbf{0.202 (0.03)} &  0.223 (0.05) \\
        SNR = 20& IR & 0.203 (0.02) & \textbf{0.221 (0.04)} \\
        \hline
        Laplacian noise & ISVR & \textbf{0.276 (0.05)}  & \textbf{0.414 (0.11)} \\
        SNR = 10& IR & 0.312 (0.05) & 0.513 (0.13) \\
        \hline
     \end{tabular}
    \end{table}

\paragraph{Real dataset}
Isotonic types of constraints can be found in different applications such as biology, ranking and weather forecast for example. Focusing on global warming type of data, reserchers have studied the anomaly of the average temperature over a year in comparison to the years 1961-1990. These temperature anomalies have a monotenous trend and keep increasing since 1850 untill 2015. Isotonic regression estimator was used on this dataset\footnote{This dataset can be downloaded from the \href{https://cdiac.ess-dive.lbl.gov/trends/temp/jonescru/jones.html}{Carbon Dioxide Information Analysis Center} at the Oak Ridge National Laboratory.} in \cite{ConstrainedLasso} and we compared our proposed ISVR estimator for anomaly prediction. The hyperparameter for the ISVR were set manually for this simulation.
\cref{weather} shows the result for the two estimators. The classical isotonic regression estimator perform better than our proposed estimator globally which is confirmed by the RMSE and MAE values of $\rmse_{IR} = 0.0067$ against $\rmse_{ISVR} = 0.083$ and $\mae_{IR} = 0.083$ against $\mae_{ISVR} = 0.116$. Howevever, we notice that in the portions where there is a significant change like between 1910-1940 and 1980-2005, the IR estimation looks like a step function whereas the ISVR estimation follows an increasing trend without these piecewise constant portions.
Note that the bias induced by the use of constraints can be overcome with reffiting methods such as~\cite{deledalle2017clear}.

\begin{figure}
    \label{weather}
    \centering
    \includegraphics[scale=0.65 ]{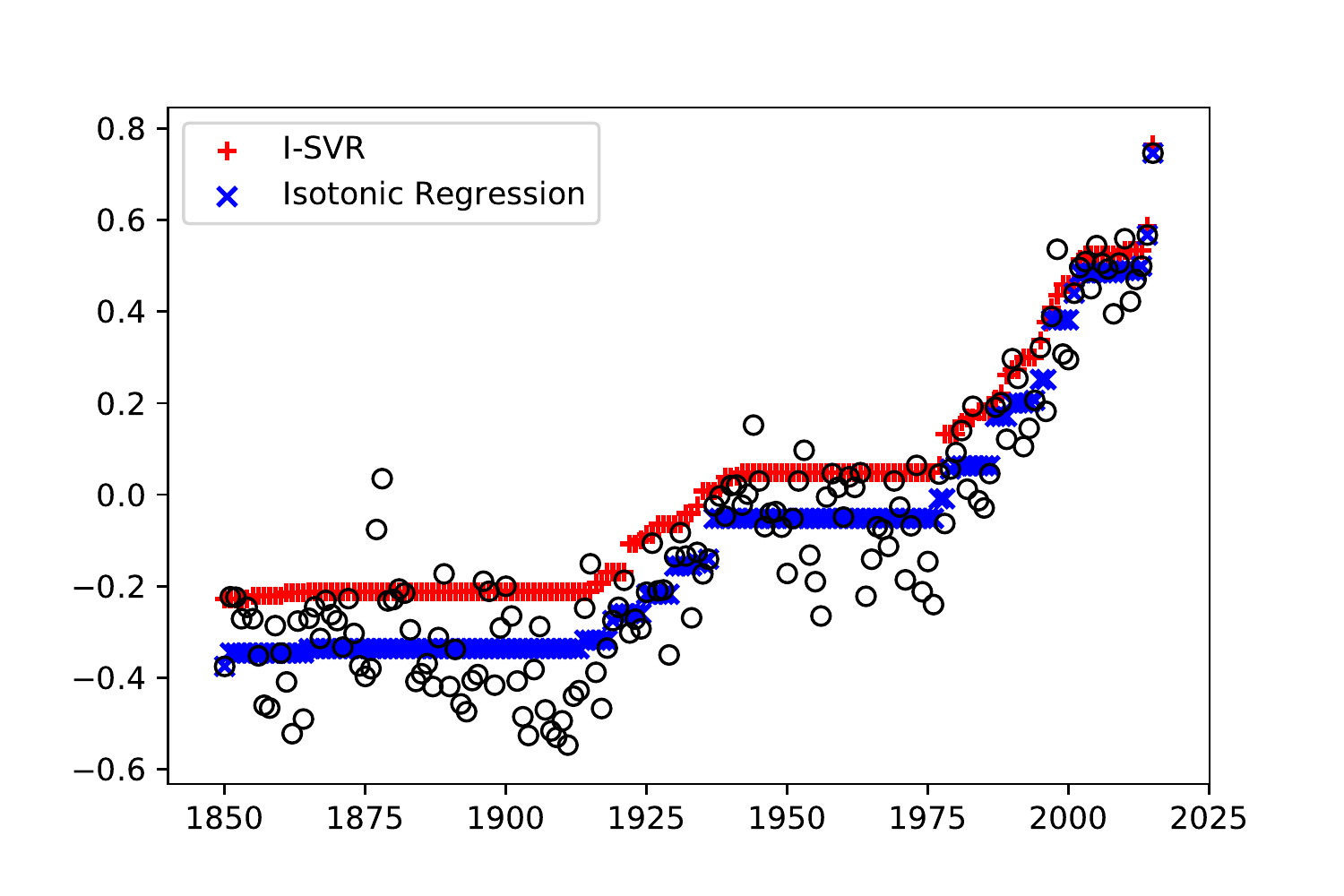}
    \caption{Global warming dataset. Annual temperature anomalies relative to 1961-1990 average, with estimated trend using Isotonic Support Vector Regression (ISVR) and the classical Isotonic Regression (IR) estimator. }
\end{figure}

\subsection{Performance of the GSMO versus SMO}
We compared the efficiency of the SMO algorithm to solve the classical SVR optimization problem and the SSVR optimization problem. To do so, we used the same data simulation process described earlier in this subsection and set the number of rows of the matrix $X$, $n = 200$ and the number of columns $p=25$. Two different settings were considered here, one without any noise in the data and another one with gaussian noise added such that the SNR would be equal to 30. The transparent trajectories represent the decrease of the objective function or the optimality score $\Delta$ for the classical SMO in blue and for the generalized SMO in red for the 50 repetitions considered. The average trajectory is represented in dense color. \cref{fig7:a} and \cref{fig7:b} are the results for the noiseless setting and  \cref{fig7:c} and \cref{fig7:d} for the setting with noise. 
When there is not noise in the data, the generalized SMO decreases faster than the classical SMO. It is important to remind that the true vector here belongs to the simplex so without any noise it is not surprising that our proposed algorithm goes faster than the classical SMO. However, when noise is adding to the data, it takes more iterations for the generalized SMO to find the solution of the optimization problem. 
\begin{figure}[t] 
    \begin{subfigure}[b]{0.5\linewidth}
      \centering
      \includegraphics[width=1.0\linewidth]{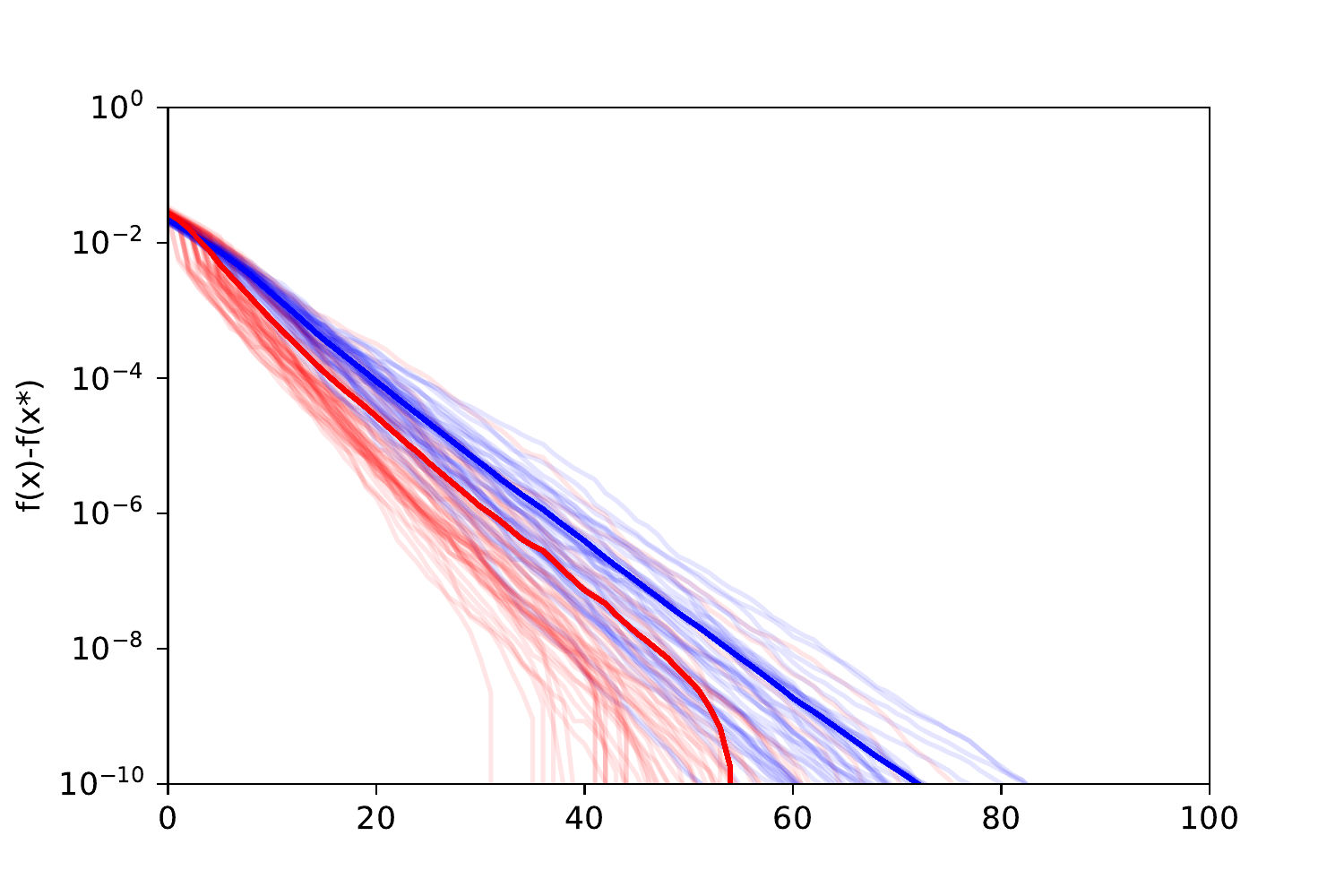} 
      \caption{Dual objective function without noise} 
      \label{fig7:a} 
      \vspace{4ex}
    \end{subfigure}%% 
    \begin{subfigure}[b]{0.5\linewidth}
      \centering
      \includegraphics[width=1.0\linewidth]{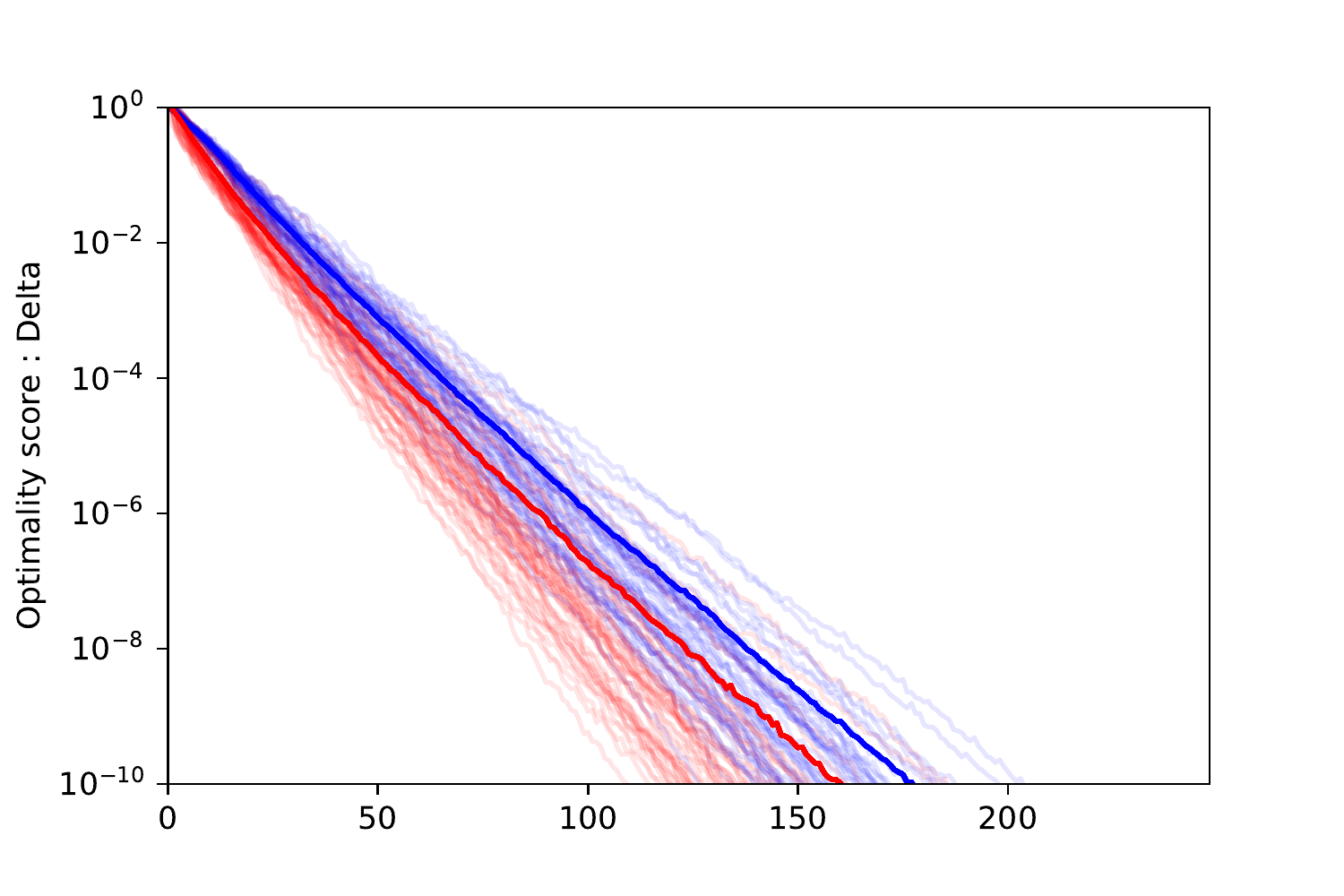} 
      \caption{Delta optimality score without noise} 
      \label{fig7:b} 
      \vspace{4ex}
    \end{subfigure} 
    \begin{subfigure}[b]{0.5\linewidth}
      \centering
      \includegraphics[width=1.0\linewidth]{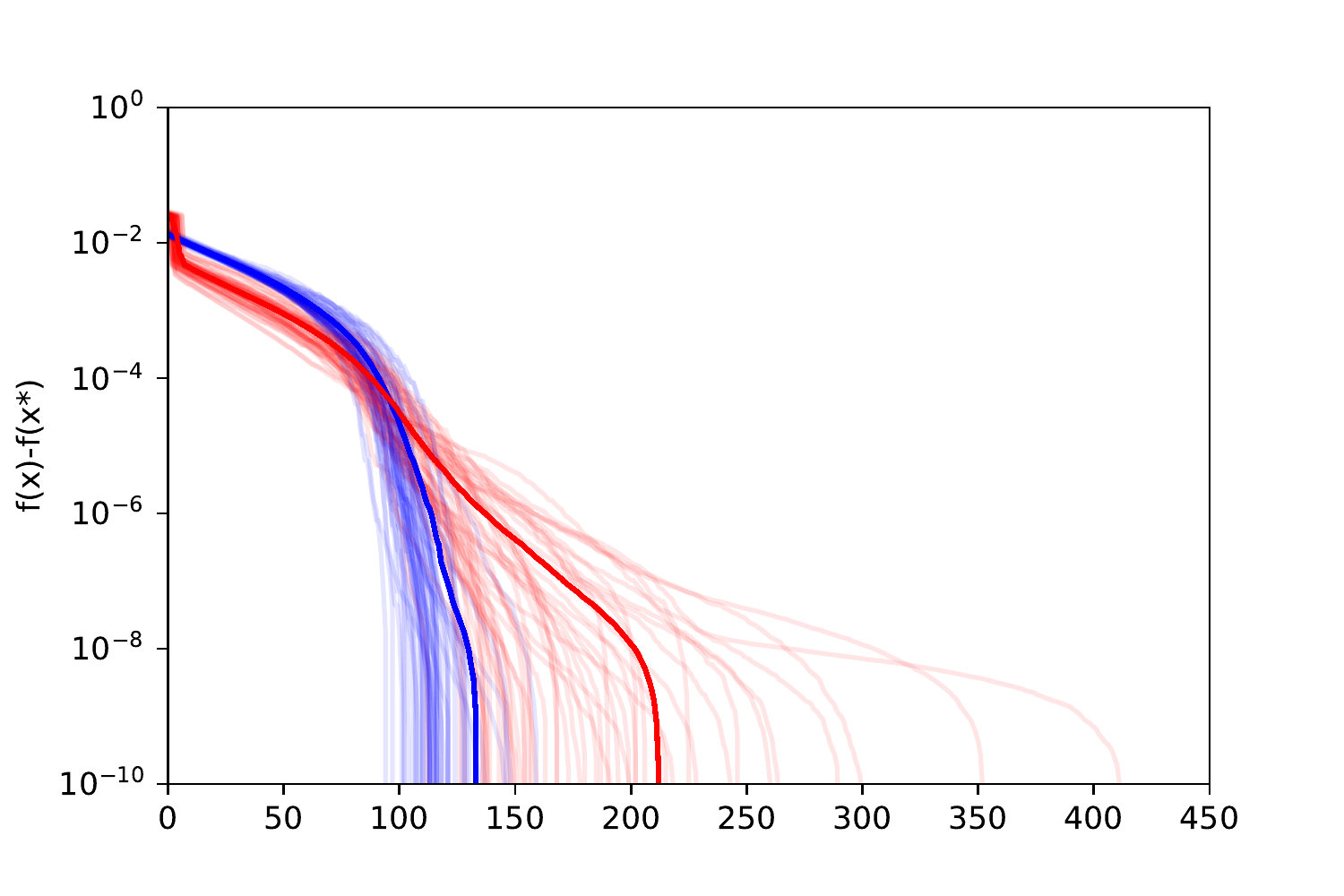} 
      \caption{Dual objective function with noise} 
      \label{fig7:c} 
    \end{subfigure}%%
    \begin{subfigure}[b]{0.5\linewidth}
      \centering
      \includegraphics[width=1.0\linewidth]{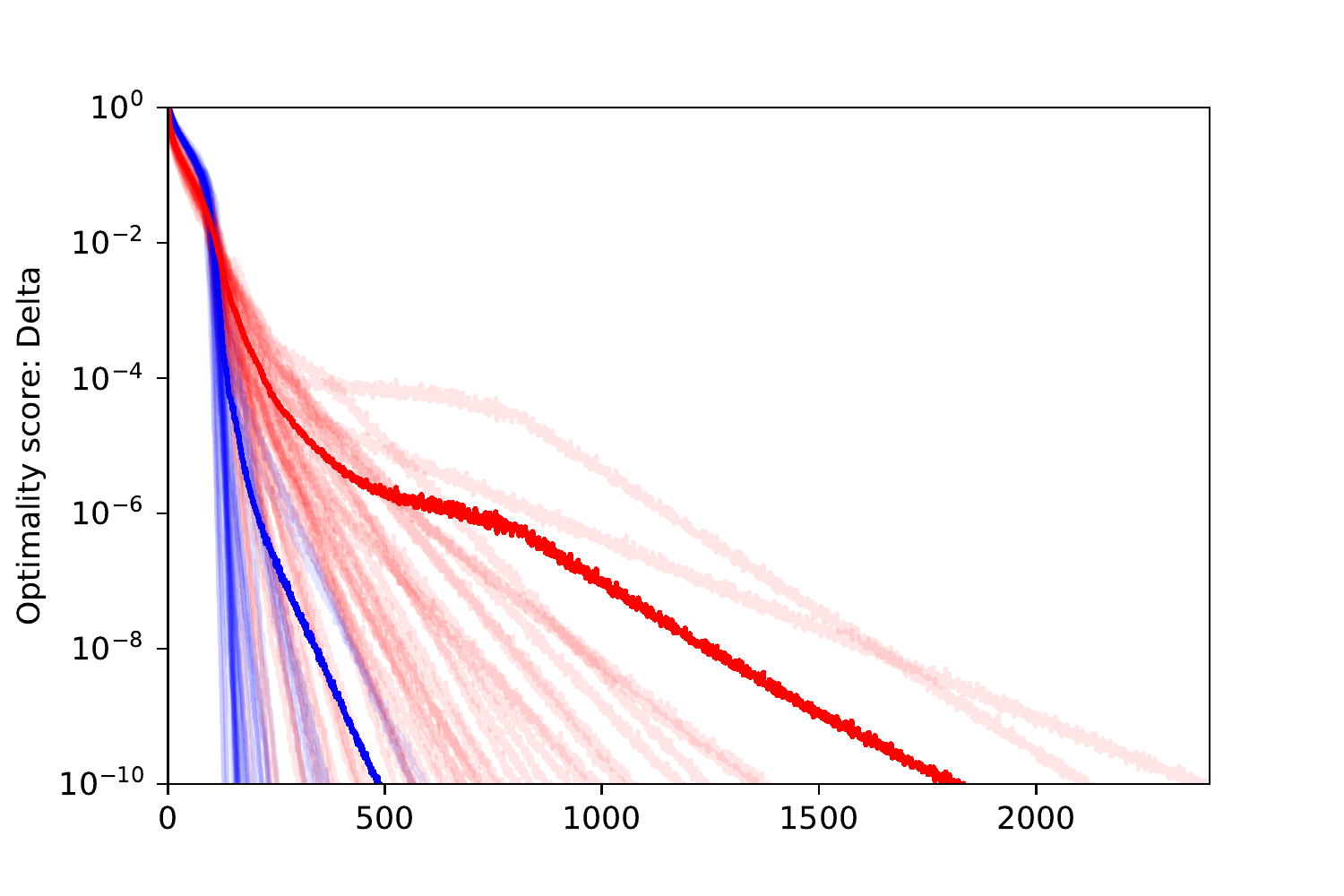} 
      \caption{Optimality score: Delta with noise} 
      \label{fig7:d} 
    \end{subfigure} 
    \caption{Plots of 50 trajectories of the dual objective function value (\cref{fig7:a}, \cref{fig7:c}) and the optimality score (\cref{fig7:b}, \cref{fig7:d}) in function of the number of iterations for the classical SMO algorithm in blue and the proposed generalized SMO in red. Two settings were used, one without noise and another one with additive gaussian noise. }
    \label{fig7} 
  \end{figure}

\section{Conclusion}
In this paper, we studied the optimization problem related to SVR with linear constraints. We showed that for this optimization problem, strong duality holds and that the dual problem is convex. We presented a generalized SMO algorithm that solve the dual problem and we proved its convergence to a solution. This algorithm uses a coordinate descent strategy where a closed form of the updates were defined. The proposed algorithm is easy to implement and shows good performance in practise. We demonstrated the good performance of our proposed estimator on different regression settings. In presence of high level of noise,  our estimator has shown to be robust and has better estimation performance in comparison to Least Squares based estimators or projected SVR estimators. 

This work leaves several open questions for future works. The question of the convergence rate of the algorithm is very natural and will have to be address in the future. Another natural question rises about the possiblity to extend our method on non-linear function estimation with linear constraints. From our point of view, it is a very challenging question because the dual optimization problem of the linearly constrained SVR loses its only dependance on the inner product between the columns of $X$, crossed terms appear in the objective function which makes it difficult to use the kernel trick as it would naturally be used for classical SVR.

\bibliographystyle{siamplain}
\bibliography{bibliography}

\newpage
\appendix
\section{Proof of \cref{dual_prop}}
\label{proof_section}
\begin{proof}
    We prove the part 2 and 3 of \cref{dual_prop} starting by writing the Lagrangian associated to Problem \cref{polyhedral_svr}:
    \begin{equation}
        \label{lagrangian}
        \begin{split}
            L= & \frac{1}{2}||\beta||^{2}+C(\nu\epsilon + \frac{1}{n}\sum_{i=1}^{n}(\xi_{i}+\xi_{i}^{*}))+\sum^{n}_{i=1} \alpha_{i}(-\epsilon-\xi_{i}-y_{i}+\beta^{T}X_{i:}+\beta_{0})\\
            + &\sum^{n}_{i=1} \alpha_{i}^{*}(-\epsilon-\xi_{i}^{*}+y_{i}-\beta^{T}X_{i:}-\beta_{0})-\sum_{i=1}^{n} \lambda_{i}\xi_{i}+\lambda_{i}^{*}\xi_{i}^{*}-\eta\epsilon \\
            + & \gamma^{T}(A\beta-b)-\mu^{T}(\Gamma\beta-d) 
        \end{split}
    \end{equation}  

We will use the notation $x_{i}^{(*)}$ to denote $x_{i}$ or $x_{i}^{*}$. From \cref{lagrangian}, we write the KKT conditions:

\begin{subequations}
    \begin{align}
         \nabla_{\beta}L&= \beta+\sum_{i=1}^{n}(\alpha_{i}-\alpha_{i}^{*})X_{i:} + A^{T}\gamma -\Gamma^{T}\mu=0  \label{eq1.1} \\
         \nabla_{\beta_{0}}L&=\sum^{n}_{i=1}\alpha_{i}-\alpha_{i}^{*} =0 \label{eq1.2} \\
         \nabla_{\xi_{i}^{(*)}}L&=\frac{C}{n}-\alpha_{i}^{(*)}-\lambda_{i}^{(*)} =0 \label{eq1.3} \\
         \nabla_{\epsilon}L&=C\nu-\sum^{n}_{i=1}\alpha_{i}+\alpha_{i}^{*}-\eta =0  \label{eq1.5} \\
         \alpha_{i}^{(*)}&\geq 0 \label{eq2.1} \\
         \eta&\geq 0 \label{eq2.2} \\
         \lambda_{i}^{(*)}&\geq 0 \label{eq2.3} \\
         \gamma_{j}&\geq 0  \label{eq2.4} \\
             \alpha_{i}(-\epsilon-\xi_{i}-y_{i}+\beta^{T}X_{i:}+\beta_{0})&=0 \label{eq3.1}  \\
            \alpha_{i}^{*}(-\epsilon-\xi_{i}^{*}+y_{i}-\beta^{T}X_{i:}-\beta_{0})&=0 \label{eq3.2} \\
         \lambda_{i}^{(*)}\xi_{i}^{(*)}&=0 \label{eq3.3} \\
         \eta\epsilon&=0 \label{eq3.4} \\
         \gamma_{j}(A\beta-b)_{j}&=0. \label{eq3.5}
    \end{align}      
\end{subequations}
From \cref{eq1.3}, we have that: 
\begin{equation}
    \lambda_{i}^{(*)}= \frac{C}{n}-\alpha_{i}^{(*)} \label{I}.
\end{equation}
From \cref{eq2.1} and \cref{eq2.3}, we have that: 
\begin{equation}
    \frac{C}{n}\geq \alpha_{i}^{*}\geq 0 \label{II}.
\end{equation}
From \cref{eq1.5}, we have that:
\begin{equation}
    \eta=C\nu - \sum_{i=1}^{n}(\alpha_{i}+\alpha_{i}^{*}) \label{III}.   
\end{equation}
From \cref{eq1.1},
\begin{equation}
    \beta= -\sum_{i=1}^{n}(\alpha_{i}-\alpha_{i}^{*})X_{i:} - A^{T}\gamma +\Gamma^{T}\mu. \label{IV}
\end{equation}
From \cref{eq1.2},
\begin{equation}
    \sum_{i=1}^{n}(\alpha_{i}-\alpha_{i}^{*})=0. \label{V}
\end{equation} 
From \cref{eq2.2}, 
\begin{equation}
    C\nu \geq \sum_{i=1}^{n}(\alpha_{i}+\alpha_{i}^{*}). \label{VI}
\end{equation}
Using \cref{I}, \cref{III}, \cref{V}, we obtain:
\begin{equation*}
    \begin{split}
        L= &\frac{1}{2}||\beta||^{2}+C\nu\epsilon + \frac{C}{n}\sum_{i=1}^{n}(\xi_{i}+\xi_{i}^{*})-\epsilon\sum^{n}_{i=1} (\alpha_{i}+\alpha_{i}^{*})-\sum^{n}_{i=1} \alpha_{i}\xi_{i}+\alpha_{i}^{*}\xi_{i}^{*}\\
    +& \sum^{n}_{i=1} (\alpha_{i}-\alpha_{i}^{*})(-y_{i}+\beta^{T}X_{i:})-\sum_{i=1}^{n} (\frac{C}{n}-\alpha_{i})\xi_{i}+(\frac{C}{n}-\alpha_{i}^{*})\xi_{i}^{*} \\
    - & (C\nu - \sum_{i=1}^{n}(\alpha_{i}+\alpha_{i}^{*}))\epsilon+\gamma^{T}(A\beta-b)-\mu^{T}(\Gamma\beta-d),
    \end{split}
\end{equation*}

and 
\begin{equation*}
    L=\frac{1}{2}||\beta||^{2}-\sum^{n}_{i=1} (\alpha_{i}-\alpha_{i}^{*})y_{i} +\sum^{n}_{i=1} (\alpha_{i}-\alpha_{i}^{*})\beta^{T}X_{i:}
   + \gamma^{T}(A\beta-b)-\mu^{T}(\Gamma\beta-d).
\end{equation*}
Replacing $\beta$ by the expression obtained in \cref{IV} yields to:

\begin{equation*}
    \begin{split}
    L= & \frac{1}{2}\langle -\sum_{i=1}^{n}(\alpha_{i}-\alpha_{i}^{*})X_{i:} - A^{T}\gamma +\Gamma^{T}\mu, -\sum_{i=1}^{n}(\alpha_{i}-\alpha_{i}^{*})X_{i:} - A^{T}\gamma +\Gamma^{T}\mu \rangle \\
    - &\sum^{n}_{i=1} (\alpha_{i}-\alpha_{i}^{*})y_{i} +\langle -\sum_{i=1}^{n}(\alpha_{i}-\alpha_{i}^{*})X_{i:} - A^{T}\gamma +\Gamma^{T}\mu ,\sum^{n}_{i=1} (\alpha_{i}-\alpha_{i}^{*}) X_{i:}\rangle \\
   + &\gamma^{T}(A(-\sum_{i=1}^{n}(\alpha_{i}-\alpha_{i}^{*})X_{i:} - A^{T}\gamma +\Gamma^{T}\mu)-b)\\
   -& \mu^{T}(\Gamma(-\sum_{i=1}^{n}(\alpha_{i}-\alpha_{i}^{*})X_{i:} - A^{T}\gamma +\Gamma^{T}\mu)-d),
    \end{split}
\end{equation*}

\begin{equation*}
\begin{split}
    L= &\frac{1}{2}\sum_{i=1}^{n}\sum^{n}_{j=1}(\alpha_{i}-\alpha_{i}^{*})(\alpha_{j}-\alpha_{j}^{*})\langle X_{i:}, X_{j:} \rangle + \frac{1}{2} \gamma^{T}AA^{T}\gamma + \frac{1}{2} \mu^{T}\Gamma\Gamma^{T}\mu \\
    + &\sum^{n}_{i=1}(\alpha_{i}-\alpha_{i}^{*})\gamma^{T}AX_{i:}-\sum^{n}_{i=1}(\alpha_{i}-\alpha_{i}^{*})\mu^{T}\Gamma X_{i:}-\gamma^{T}A\Gamma^{T}\mu-\sum^{n}_{i=1} (\alpha_{i}-\alpha_{i}^{*})y_{i} \\
    -&\sum_{i=1}^{n}\sum^{n}_{j=1}(\alpha_{i}-\alpha_{i}^{*})(\alpha_{j}-\alpha_{j}^{*})\langle X_{i:}, X_{j:} \rangle -\sum^{n}_{i=1}(\alpha_{i}-\alpha_{i}^{*})\gamma^{T}AX_{i:} \\
    + & \sum^{n}_{i=1}(\alpha_{i}-\alpha_{i}^{*})\mu^{T}\Gamma X_{i:}-\sum^{n}_{i=1}(\alpha_{i}-\alpha_{i}^{*})\gamma^{T}AX_{i:}-\gamma^{T}AA^{T}\gamma + \gamma^{T}A\Gamma^{T}\mu - \gamma^{T}b \\+ & \sum^{n}_{i=1}(\alpha_{i}-\alpha_{i}^{*})\mu^{T}\Gamma X_{i:}-\mu^{T}\Gamma\Gamma^{T}\mu + \gamma^{T}A\Gamma^{T}\mu + \mu^{T} d, \\
\end{split}
\end{equation*}
and finally 
\begin{equation}
\begin{split}
    L= & -\frac{1}{2}(\alpha-\alpha^{*})^{T}Q(\alpha-\alpha^{*}) - \frac{1}{2} \gamma^{T}AA^{T}\gamma - \frac{1}{2} \mu^{T}\Gamma\Gamma^{T}\mu- \sum^{n}_{i=1}(\alpha_{i}-\alpha_{i}^{*})\gamma^{T}AX_{i:} \\
    + &\sum^{n}_{i=1}(\alpha_{i}-\alpha_{i}^{*})\mu^{T}\Gamma X_{i:}+\gamma^{T}A\Gamma^{T}\mu-\sum^{n}_{i=1} (\alpha_{i}-\alpha_{i}^{*})y_{i}-\gamma^{T}b+\mu^{T} d.
\end{split}
\label{last_lagrangian}
\end{equation}
Using the constraints derived from \cref{eq2.4}, \cref{V}, \cref{VI}, \cref{II} and the expression of the Lagragian \cref{last_lagrangian}, the dual problem is as follows: 

\begin{eqnarray}
    \begin{aligned}
    & \underset{\alpha,\alpha^{*},\gamma,\mu}{\min}
    & \frac{1}{2}((\alpha-\alpha^{*})^{T}Q(\alpha-\alpha^{*})+\gamma^{T}AA^{T}\gamma +\mu^{T}\Gamma\Gamma^{T}\mu+2\sum^{n}_{i=1}(\alpha_{i}-\alpha_{i}^{*})\gamma^{T}AX_{i:}\\
    & &-2\sum^{n}_{i=1}(\alpha_{i}-\alpha_{i}^{*})\mu^{T}\Gamma X_{i:}-2\gamma^{T}A\Gamma^{T}\mu)+\sum^{n}_{i=1} (\alpha_{i}-\alpha_{i}^{*})y_{i}+\gamma^{T}b-\mu^{T}d\\
    & \text{subject to}
    & 0 \leq \alpha_{i}^{(*)}\leq \frac{C}{n} \\
    & & \sum_{i=1}^{n}\alpha_{i}+\alpha_{i}^{*}\leq C\nu \\
    & & \sum_{i=1}^{n}\alpha_{i}-\alpha_{i}^{*}=0 \\
    & & \gamma_{j}\geq 0. \\
    \end{aligned}
    \label{dual_polyhedral_svr_proof}
\end{eqnarray}

The equation linking the primal and the dual optimization problems is given by \cref{IV} which finishes the proof. 
\end{proof}

\section{Proof of \cref{prop_structure}}
\label{proof_structure}
\begin{proof}
    To prove part 1., let's recall that $\alpha_{i}$, $\alpha_{i}^{*}$ are the lagrange multipliers associated to the optimization problem \cref{polyhedral_svr} constraints: 
        \begin{equation}
           \begin{aligned}
               \beta^{T}X_{i:}+\beta_{0}-y_{i}\leq \epsilon+\xi_{i} \\
               y_{i}-\beta^{T}X_{i:}-\beta_{0}\leq \epsilon+\xi_{i}^{*}. \\
           \end{aligned} 
        \end{equation}
        The complementary optimality conditions leads to 
        \begin{equation*}
            \begin{aligned}
                \alpha_{i}(\beta^{T}X_{i:}+\beta_{0}-y_{i}-\epsilon-\xi_{i})=0 \\
                \alpha_{i}^{*}(y_{i}-\beta^{T}X_{i:}-\beta_{0}-\epsilon-\xi_{i}^{*})=0. \\
            \end{aligned} 
         \end{equation*}
        Let's now suppose that $\alpha_{i}>0$ and $\alpha_{i}^{*}>0$ which implies that  
         \begin{equation*}
            \begin{aligned}
            \beta^{T}X_{i:}+\beta_{0}-y_{i}-\epsilon-\xi_{i}=0 \\
            y_{i}-\beta^{T}X_{i:}-\beta_{0}-\epsilon-\xi_{i}^{*}=0.
            \end{aligned}
         \end{equation*}
         It follows that $-2\epsilon=\xi_{i}+\xi_{i}^{*}$ and $\xi_{i},\xi_{i}^{*}\geq 0$ which implies $\xi_{i}=\xi_{i}^{*}=\epsilon=0$. This goes against our condition $\epsilon>0$. \\

         To prove part 2., we need to remind the optimality conditions given in \cref{proof_section}, \cref{eq3.4} and \cref{III} leads to
         \begin{equation*}
             (C\nu-\sum_{i=1}^{l}\alpha_{i}+\alpha_{i}^{*})\epsilon=0.
         \end{equation*}

         Thus, if $\epsilon>0$ we have that $\sum_{i=1}^{n} \alpha_{i}+\alpha_{i}^{*}=C\nu$. 
\end{proof}

\section{Proof of \cref{prop_non_violating}}
\label{proof_non_violating}

We start by giving a lemma that will be usefull to prove the proposition for the blocks $\alpha$ and $\alpha^{*}$. 
\begin{lemma}
    \label{lemma_gradient}
    If the update between iteration $k$ and $k+1$ happens in the block $\alpha$ (or $\alpha^{*}$) and that $(i,j)$ is the most violating pair of variables then 
    $$
    \nabla_{\alpha_{i}}f(\theta^{k+1})-\nabla_{\alpha_{j}}f(\theta^{k+1}) = \nabla_{\alpha_{i}}f(\theta^{k})-\nabla_{\alpha_{j}}f(\theta^{k}) + t^{*}
    (Q_{ii}+Q_{jj}-2Q_{ij})
$$
\end{lemma}
\begin{proof}
    Let's recall that the update in the block $\alpha$ (or $\alpha^{*}$) has the following form 
    \begin{equation*}
        \begin{aligned}
            \alpha_{i}^{k+1} = \alpha_{i}^{k}+ t^{*} \\
            \alpha_{j}^{k+1} = \alpha_{j}^{k}- t^{*},
        \end{aligned}
    \end{equation*}
    with $t^{*}$ as defined in \cref{def_updates}. In a stacked form we have that
    \begin{equation*}
        \begin{split}
            \nabla_{\alpha_{i}}f(\theta^{k+1})-\nabla_{\alpha_{j}}f(\theta^{k+1})& = (Q\theta^{k+1})_{i}+l_{i}-(Q\theta^{k+1})_{j}-l_{j} \\
            & = \sum_{s=1}^{2n+k_{1}+k_{2}} Q_{is}\theta^{k+1}_{s} + l_{i} -\sum_{s=1}^{2n+k_{1}+k_{2}} Q_{js}\theta^{k+1}_{s}- l_{j} \\
            & = \nabla_{\alpha_{i}}f(\theta^{k})-\nabla_{\alpha_{j}}f(\theta^{k}) + t^{*}(Q_{ii}-Q_{ij})+t^{*}(Q_{jj}-Q_{ij}) \\
            & = \nabla_{\alpha_{i}}f(\theta^{k})-\nabla_{\alpha_{j}}f(\theta^{k}) + t^{*}(Q_{ii}+Q_{jj}-2Q_{ij}). 
        \end{split}
    \end{equation*}
    \hspace{0.5em}
\end{proof}

This lemma is helpful for the proof the blocks $\gamma$ and $\mu$.
\begin{lemma}
    \label{gradient_update}
    If $\theta_{i}$ is the updated variable at iteration $k$, then the following holds:
    \begin{equation*}
        \nabla_{i}f(\theta^{k+1})=\bar{Q}_{ii}(\theta_{i}^{k+1}-\theta_{i}^{k})+\nabla_{i}f(\theta^{k})
    \end{equation*}
\end{lemma}

\begin{proof}
    The proof is straightforward,
    \begin{equation*}
        \begin{split}
            \nabla_{i}f(\theta^{k+1}) = & (\bar{Q}\theta_{i}^{k+1})_{i}+l_{i} \\
            = & \sum_{s=1}^{2n+k_{1}+k_{2}} \bar{Q}_{is}\theta_{s}^{k+1}+l_{i} \\
            = & \sum_{s\neq i}^{2n+k_{1}+k_{2}} \bar{Q}_{is}\theta_{s}^{k+1}+l_{i}+\bar{Q}_{ii}\theta_{i}^{k+1} \\
            = & \nabla_{i}f(\theta_{k}) +\bar{Q}_{ii}\theta_{i}^{k+1} -\bar{Q}_{ii}\theta_{i}^{k} \\
            = & \bar{Q}_{ii}(\theta_{i}^{k+1}-\theta_{i}^{k})+\nabla_{i}f(\theta^{k}).
        \end{split}
    \end{equation*}
    \hspace{1em}
\end{proof}

Let's now give the proof of \cref{prop_non_violating}.
\begin{proof}
    Let's consider that the update between iteration $k$ and $k+1$ takes place in the block $\alpha$. We will define $(i,j)$ as the most violating pair of variables as defined in \cref{SMO}. From the discussion in \cref{updates}, we know that minimizing the objective function of \cref{dual_polyhedral_svr} considering that only the parameter $t$ is a variable leads to minimizing the following function:
    \begin{equation}
        \psi (t)=\frac{1}{2}t^{2}(\bar{Q}_{ii}+\bar{Q}_{jj}-2\bar{Q}_{ij})+t(\nabla_{\alpha_{i}}f(\theta^{k})-\nabla_{\alpha_{j}}f(\theta^{k}))+K
        \label{psi_t}
    \end{equation}
    We recall that $t$ is the parameter that will be used for the update of $\alpha_{i}$ and $\alpha_{j}$ and $K$ is a constant term. We also have the following result from \cref{lemma_gradient}:
    \begin{equation}
        \nabla_{\alpha_{i}}f(\theta^{k+1})-\nabla_{\alpha_{j}}f(\theta^{k+1})=\nabla_{\alpha_{i}}f(\theta^{k})-\nabla_{\alpha_{j}}f(\theta^{k})+(\bar{Q}_{ii}+\bar{Q}_{jj}-2\bar{Q}_{ij})t^{*}
        \label{eq_gradient}
    \end{equation}

    The minimization update takes place in the square $S=[0, \frac{C}{n}]\times [0,\frac{C}{n}]$ illustrated in \cref{square}. 
    
    \begin{figure}
        \centering
        \includegraphics[scale=0.3 ]{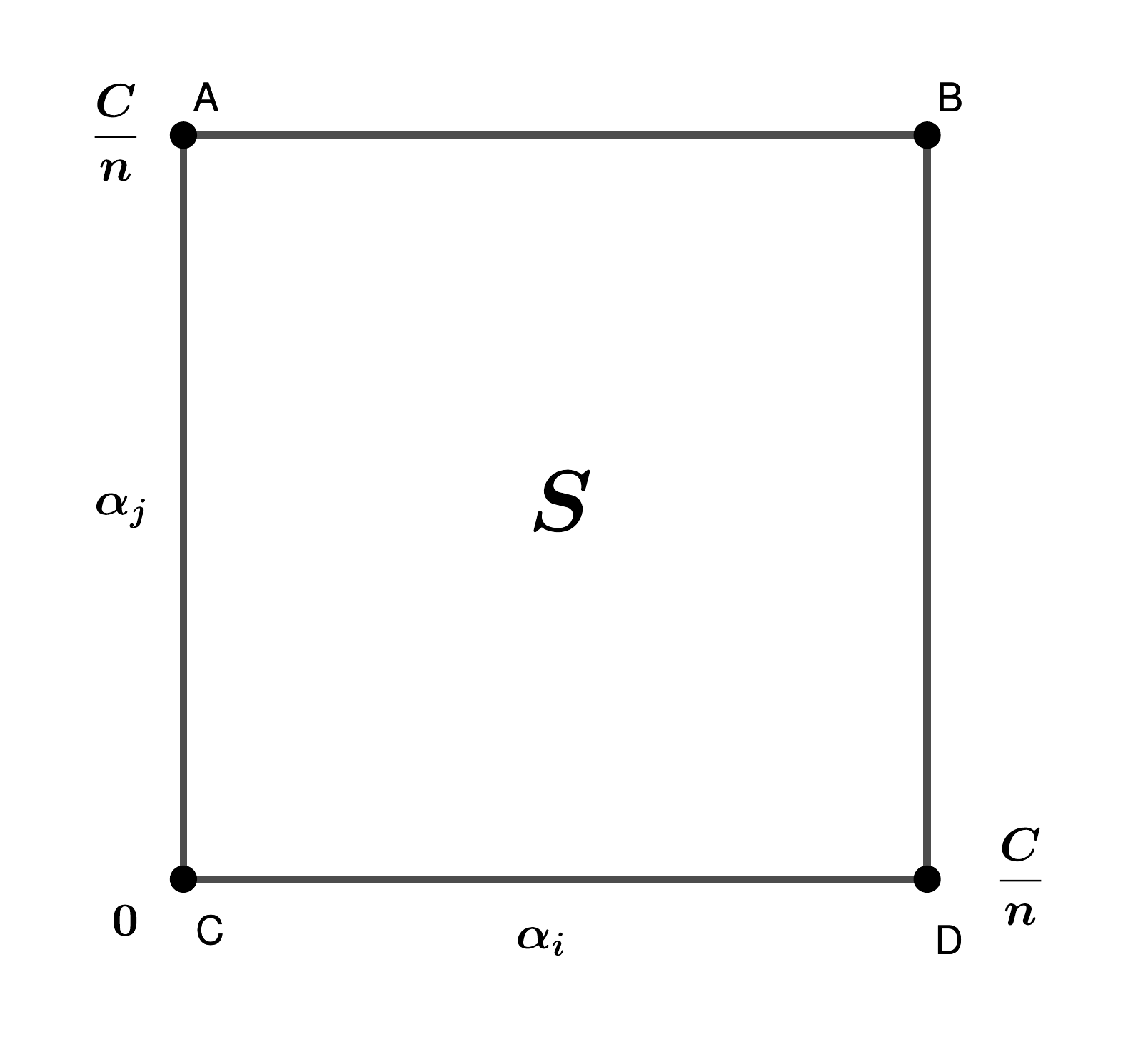}
        \caption{Possible update for the block $\alpha$ or $\alpha^{*}$}
        \label{square}
    \end{figure}
At points B and C of the square S, $(i,j)$ cannot be a $\tau$-violating pair of variables because they belong to the same set of indices $I_{\text{up}}$ (or $I_{\text{low}}$). Everywhere else, violation can take place. 

\begin{itemize}
    \item On $]CA]$, $\alpha_{i}=0$ and $\alpha_{j}>0$ so $i\in I_{\text{up}}$ and $j\in I_{\text{low}}$ which means that by definition of $\tau$-violating pair of variable 
    $$
    \nabla_{\alpha_{i}}f(\theta^{k})-\nabla_{\alpha_{j}}f(\theta^{k})<-\tau<0 
    $$
    which means $t_{q} =\frac{-(\nabla_{\alpha_{i}} f(\theta^{k})-\nabla_{\alpha_{j}} f(\theta^{k}))}{(\bar{Q}_{ii}+\bar{Q}_{jj}-2\bar{Q}_{ij})}>0$. Let's remind that : 
    \begin{equation}
        \max(-\alpha_{i},\alpha_{j}-\frac{C}{n})\leq t^{*} \leq \min(\frac{C}{n}-\alpha_{i},\alpha_{j})
        \label{t_inequality}
    \end{equation}

    It means that on $]CA]$, \cref{t_inequality} becomes : $0\leq t^{*} \leq \alpha_{j}$
    There are then two possibilities:
        \begin{itemize}
            \item if $t_{q}\geq \alpha_{j}$, it implies because of the constraints on $t^{*}$, that $t^{*}=\alpha_{j}$. The update becomes then $\alpha_{i}^{k+1}=\alpha^{k}_{i}+\alpha^{k}_{j}$ and $\alpha^{k+1}_{j}=0$. Then $j$ belongs to the set of indices $I_{\text{up}}$ and $i$ belongs to $I_{\text{low}}$.  
            From \cref{eq_gradient}, we deduce that $\nabla_{\alpha_{i}}f(\theta^{k+1})-\nabla_{\alpha_{j}}f(\theta^{k+1})\leq 0$ which proves that $(i,j)$ is not a violating pair of variable anymore and that $\alpha^{k+1}\neq \alpha^{k}$
            \item Second possibility is that $t_{q}\leq \alpha_{j}$ then $t^{*}=t_{q}$, then $(\alpha^{k+1}_{i},\alpha^{k+1}_{j})$ belongs to $\text{int}(S)$. From \eqref{eq_gradient}, we deduce that $\nabla_{\alpha_{i}}f(\theta^{k+1})-\nabla_{\alpha_{j}}f(\theta^{k+1})=0 $, $(i,j)$ is not a $\tau$-violating pair of variables anymore and $\alpha^{k+1}\neq \alpha^{k}$.
        \end{itemize}

        \item On $]CD]$, $\alpha_{i}>0$ and $\alpha_{j}=0$ so $i\in I_{\text{low}}$ and $j\in I_{\text{up}}$ which means that by definition of $\tau$-violating pair of variable 
    $$
    \nabla_{\alpha_{i}}f(\theta^{k})-\nabla_{\alpha_{j}}f(\theta^{k})>\tau 
    $$
    which yields to $t_{q}<0$. 
    It means that on $]CD]$, \cref{t_inequality} becomes : $-\alpha_{i}\leq t^{*} \leq 0$
    There are then two possibilities:
        \begin{itemize}
            \item $t_{q}\leq -\alpha_{i}$, it implies because of the constraints on $t^{*}$, that $t^{*}=-\alpha_{i}$. The update becomes then $\alpha^{k+1}_{i}=0$ and $\alpha^{k+1}_{j}=\alpha^{k}_{i}$. Then $j$ belongs to the set of indices $I_{\text{low}}$ and $i$ belongs to $I_{\text{up}}$ at $\alpha^{k+1}$.  
            From \cref{eq_gradient}, we deduce that $\nabla_{\alpha_{i}}f(\theta^{k+1})-\nabla_{\alpha_{j}}f(\theta^{k+1})\geq 0$ which proves that $(i,j)$ is not a violating pair of variable anymore and that $\alpha^{k+1}\neq \alpha^{k}$
            \item Second possibility is that $t_{q}\geq -\alpha_{i}$ then $t^{*}=t_{q}$. Implying that $(\alpha^{k+1}_{i},\alpha^{k+1}_{j})$ belongs to $\text{int}(S)$. From \cref{eq_gradient}, we deduce that $\nabla_{\alpha_{i}}f(\theta^{k+1})-\nabla_{\alpha_{j}}f(\theta^{k+1})=0 $, $(i,j)$ is not a $\tau$-violating pair of variables anymore and $\alpha^{k+1}\neq \alpha^{k}$.
        \end{itemize}

\item On $[AB[$, $0\leq \alpha_{i}<\frac{C}{n}$ and $\alpha_{j}=\frac{C}{n}$ so $i\in I_{\text{up}}$ and $j\in I_{\text{low}}$ which means that by definition of $\tau$-violating pair of variable 
    $$
    \nabla_{\alpha_{i}}f(\theta^{k})-\nabla_{\alpha_{j}}f(\theta^{k})<-\tau<0, 
    $$
    which implies $t_{q}>0$. 

    It means that on $[AB[$, \cref{t_inequality} becomes $0\leq t^{*} \leq \frac{C}{l}-\alpha_{i}.$
    There are then two possibilities:
        \begin{itemize}
            \item if $t_{q}\geq \frac{C}{l}-\alpha_{i}$, it implies, because of the constraints on $t^{*}$, that $t^{*}=\frac{C}{n}-\alpha_{i}$. The update is $\alpha^{k+1}_{i}=\frac{C}{n}$ and $\alpha^{k+1}_{j}=\alpha_{i}^{k}$. Then $j$ belongs to the set of indices $I_{\text{up}}$ and $i$ belongs to $I_{\text{low}}$.  
            From \cref{eq_gradient} we deduce that $ \nabla_{\alpha_{i}}f(\theta^{k+1})-\nabla_{\alpha_{j}}f(\theta^{k+1})\leq 0$ which proves that $(i,j)$ is not a violating pair of variable anymore and that $\alpha^{k+1}\neq \alpha^{k}$.
            \item Second possibility is that $t_{q}\leq \frac{C}{n}-\alpha_{i}$. Thus $t^{*}=t_{q}$, then $(\alpha^{k+1}_{i},\alpha^{k+1}_{j})$ belongs to $\text{int}(S)$. From \cref{eq_gradient}, we deduce that $\nabla_{\alpha_{i}}f(\theta^{k+1})-\nabla_{\alpha_{j}}f(\theta^{k+1})=0 $, $(i,j)$ is not a $\tau$-violating pair of variables anymore and $\alpha^{k+1}\neq \alpha^{k}$.
        \end{itemize}

        \item On $]BD]$, $\alpha_{i}=\frac{C}{n}$ and $0\leq\alpha_{j}<\frac{C}{l}$. Thus we have that $i\in I_{\text{low}}$ and $j\in I_{\text{up}}$ which means that by definition of $\tau$-violating pair of variable 
        $$
        \nabla_{\alpha_{i}}f(\theta^{k})-\nabla_{\alpha_{j}}f(\theta^{k})>\tau, 
        $$
        which yields to $t_{q}<0$. 
        It means that on $]BD]$, \cref{t_inequality} becomes $\alpha_{j}-\frac{C}{n}\leq t^{*} \leq 0.$
        There are then two possibilities:
            \begin{itemize}
                \item if $t_{q}\leq \alpha_{j}-\frac{C}{n}$, it implies that $t^{*}=\alpha_{j}-\frac{C}{n}$. The update becomes $\alpha^{k+1}_{i}=\alpha^{k}_{j}$ and $\alpha^{k+1}_{j}=\frac{C}{n}$. Then $j$ belongs to the set of indices $I_{\text{low}}$ and $i$ belongs to $I_{\text{up}}$ at $\alpha^{k+1}$.  
                From \cref{eq_gradient}, we deduce that $\nabla_{\alpha_{i}}f(\theta^{k+1})-\nabla_{\alpha_{j}}f(\theta^{k+1})\geq 0$ which proves that $(i,j)$ is not a violating pair of variable anymore and that $\alpha^{k+1}\neq \alpha^{k}$.
                \item Second possibility is that $t_{q}\geq \alpha_{j}-\frac{C}{n}$. Then $t^{*}=t_{q}$ and $(\alpha^{k+1}_{i},\alpha^{k+1}_{j})\in int(S)$. From \cref{eq_gradient}, we deduce that $\nabla_{\alpha_{i}}f(\theta^{k+1})-\nabla_{\alpha_{j}}f(\theta^{k+1})=0 $, $(i,j)$ is not a $\tau$-violating pair of variables anymore and $\alpha^{k+1}\neq \alpha^{k}$.
            \end{itemize} 
        \item Inside the square $S$, if $i\in I_{\text{low}}$ and $j\in I_{\text{up}}$, we have that $t_{q}<0$. Then there will be three possibilities for the update coming from this inequality $\max(-\alpha_{i},\alpha_{j}-\frac{C}{n})\leq t^{*} < 0$. The same discussion as the one we had for the edges of S gives the desired results, the only difference here is that there are three different possibilities: 2 clipped updates possibilities and the update using $t_{q}$. 
        The same observation is true for the case where $i\in I_{\text{up}}$ and $j\in I_{\text{low}}$, it will only change the sign of $t_{q}$. Thus changing the 2 possible clipped update using the upper bound of \cref{t_inequality} or the update using $t_{q}$. Everything leads to the conclusion that $(i,j)$ cannot be a violating pair of variables at iteration $k+1$ and that $\alpha^{k+1}\neq \alpha$.     
\end{itemize}

The same arguments are used to prove the same for the block $\alpha^{*}$, the proof is similar. 

    Let's now prove that when the update takes place at index $i$ in the block $\gamma$ then $i$ is not violating variable at iteration $k+1$. Then we need to show that $\nabla_{\gamma_{i}}f(\theta^{k+1})\geq 0$. Let's start with the case where the update $\gamma_{i}^{k+1}=\frac{\nabla_{\gamma_{i}}f(\theta^{k})}{\bar{Q}_{ii}}-\gamma^{k}_{i}$. Using \cref{gradient_update}, we have that $\nabla_{\gamma_{i}}f(\theta^{k+1})=0$. 
    The second possible case is $\gamma_{i}^{k+1}=0$ because $-\frac{\nabla_{\gamma_{i}}f(\theta^{k})}{\bar{Q}_{ii}}+\gamma^{k}_{i}\leq 0$. If $\gamma_{i}^{k+1}=0$ then $\nabla_{\gamma_{i}}f(\theta^{k+1})=-\bar{Q}_{ii}\gamma_{i}^{k}+\nabla_{\gamma_{i}}f(\theta^{k})$. $\bar{Q}_{ii}$ is positive because it is a diagonal element of a Gram matrix $(A^{T}A)$ thus we get that $\nabla_{\gamma_{i}}f(\theta^{k+1})\geq 0$, which proves that $i$ is not a violating variable anymore.

    The proof for the block $\mu$ relies on the same idea except that it is simpler because there is no clipped updates possible so $\nabla_{\mu_{i}}f(\theta^{k+1})=0$ if the updates takes place at $\mu_{i}$ which also proves that $i$ is not a violating variable for this block of variables anymore. 
\end{proof}

\section{Proof of \cref{convergence_theorem}}
\label{proof_theorem}

We begin the proof of the theorem by giving several preliminary results that will be hepful for giving the final proof. 
The first result gives a bound for controlling the distance of the primal iterates generated by the algorithm and the solution of \cref{polyhedral_svr}.
\begin{lemma}
    \label{lemma_bound}
    For any SMO-LSSVR iterate $\beta^{k}=-\sum_{i=1}^{n} (\alpha_{i}^{k}-(\alpha_{i}^{*})^{k})X_{i:}-A^{T}\gamma^{k}+\Gamma^{T}\mu^{k}$, $\beta^{\text{opt}}$ a solution of \cref{polyhedral_svr} and $\theta^{\text{opt}}$ a solution of \cref{dual_polyhedral_svr}, it holds that 
    \begin{equation*}
        \frac{1}{2}||\beta^{k}-\beta^{\text{opt}}||\leq f(\theta^{k})-f(\theta^{\text{opt}}).
    \end{equation*}
\end{lemma}

\begin{proof}

    A first observation is that the relationship between the primal optimization problem and the dual leads to this equality 
    \begin{equation}
        f(\theta^{k})=\frac{1}{2}||\beta^{k}||^{2}+l^{T}\theta^{k}.
        \label{equality_dual_primal}
    \end{equation}

    Replacing $\beta^{k}$ by $-\sum_{i=1}^{n} (\alpha_{i}^{k}-(\alpha_{i}^{*})^{k})X_{i:}-A^{T}\gamma^{k}+\Gamma^{T}\mu^{k}$ leads to \cref{equality_dual_primal}.
    We have already seen that there is strong duality between both problems so the dual gap is zero at the solutions. Thus it means that for any primal optimal solution $(\beta^{\text{opt}},\beta_{0}^{\text{opt}},\xi^{\text{opt}},\xi^{\text{opt}},\epsilon^{\text{opt}})$ and any dual solution $\theta^{\text{opt}}$, it holds true that 
    $$
        \frac{1}{2} ||\beta^{\text{opt}}||^{2}+C(\nu\epsilon^{\text{opt}}+\frac{1}{n}\sum_{i=1}^{n}\xi_{i}^{\text{opt}}+\xi_{i}^{\text{opt}})=-f(\theta^{\text{opt}})=-\frac{1}{2}||\beta^{\text{opt}}||^{2}-l^{T}\theta^{\text{opt}}.
    $$

    Using the equation link between primal and dual yields to
    \begin{align*} 
        \langle \beta, \beta^{\text{opt}}\rangle &= \langle -\sum_{i=1}^{n} (\alpha_{i}-\alpha_{i}^{*})X_{i:}-A^{T}\gamma+\Gamma^{T}\mu, \beta^{\text{opt}}\rangle  \\ 
        &= -\langle A^{T}\gamma,\beta^{\text{opt}} \rangle -\sum_{i=1}^{n} (\alpha_{i}-\alpha_{i}^{*}) \langle X_{i:}, \beta^{\text{opt}}\rangle + \langle \Gamma^{T}\mu,\beta^{\text{opt}}\rangle.
        \end{align*}
        Since $\sum_{i=1}^{n} (\alpha_{i}-\alpha_{i}^{*})=0$, we have that
    \begin{align*}
        \langle \beta, \beta^{\text{opt}}\rangle &= -\langle A^{T}\gamma,\beta^{\text{opt}} \rangle -\sum_{i=1}^{n} (\alpha_{i}-\alpha_{i}^{*}) \langle X_{i:}, \beta^{\text{opt}}\rangle + \langle \Gamma^{T}\mu,\beta^{\text{opt}}\rangle  -\beta_{0}^{\text{opt}}\sum_{i=1}^{n} (\alpha_{i}-\alpha_{i}^{*}) \\
        &= -\langle A^{T}\gamma,\beta^{\text{opt}} \rangle -\sum_{i=1}^{n} \alpha_{i}(\langle X_{i:}, \beta^{\text{opt}}\rangle + \beta_{0}^{\text{opt}})+\sum_{i=1}^{n} \alpha_{i}^{*}(\langle X_{i:}, \beta^{\text{opt}}\rangle + \beta_{0}^{\text{opt}})\\ 
        & +\langle \Gamma^{T}\mu,\beta^{\text{opt}}\rangle.
    \end{align*}
    Moreover, using the constraints of \cref{polyhedral_svr} and the fact that $\alpha\geq 0$ and $\alpha^{*}\geq 0$ it holds that:
    \begin{align*}
        \langle \beta, \beta^{\text{opt}}\rangle & \geq -\langle A^{T}\gamma,\beta^{\text{opt}} \rangle + \sum_{i=1}^{n} \alpha_{i} (-y_{i}-\epsilon^{\text{opt}}-\xi_{i}^{\text{opt}})+\sum_{i=1}^{n} \alpha_{i}^{*} (y_{i}-\epsilon^{\text{opt}}-(\xi_{i}^{*})^{\text{opt}})\\ & +\langle \Gamma^{T}\mu,\beta^{\text{opt}}\rangle  \\
        &= -\langle A^{T}\gamma,\beta^{\text{opt}} \rangle - \sum_{i=1}^{n} (\alpha_{i}-\alpha_{i}^{*})y_{i} -\epsilon^{\text{opt}}C\nu- \sum_{i=1}^{n}\alpha_{i}\xi_{i}^{\text{opt}}+\alpha_{i}^{*}(\xi_{i}^{*})^{\text{opt}}\\
        & +\langle \Gamma^{T}\mu,\beta^{\text{opt}}\rangle 
    \end{align*}

    Finally we have
    \begin{align*}
        \frac{1}{2}||\beta^{k}-\beta^{\text{opt}}||^{2} &= \frac{1}{2}|| \beta^{k}||^{2}-\langle \beta^{k}, \beta^{\text{opt}}\rangle + \frac{1}{2} ||\beta^{\text{opt}}||^{2}\\
        & \leq \frac{1}{2}|| \beta^{k}||^{2} +\langle A^{T}\gamma^{k},\beta^{\text{opt}} \rangle + \sum_{i=1}^{n} (\alpha_{i}^{k}-(\alpha_{i}^{*})^{k})y_{i} +\epsilon^{\text{opt}}C\nu \\
        & + \sum_{i=1}^{n}\alpha_{i}^{k}\xi_{i}^{\text{opt}}+(\alpha_{i}^{*})^{k}(\xi_{i}^{*})^{\text{opt}}-\langle \Gamma^{T}\mu^{k},\beta^{\text{opt}}\rangle +\frac{1}{2} ||\beta^{\text{opt}}||^{2} 
    \end{align*}

    Since $\beta^{\text{opt}}$ statisfies the constraints of the primal optimization problem, it holds that $
     \langle \Gamma^{T}\mu,\beta^{\text{opt}}\rangle=\mu^{T}d
    $
    and since $\gamma\geq 0$ we have $\langle A^{T}\gamma,\beta^{\text{opt}}\rangle \leq \gamma^{T}b$, thus
    \begin{align*}
        \frac{1}{2}||\beta^{k}-\beta^{\text{opt}}||^{2} & \leq \frac{1}{2}|| \beta^{k}||^{2} +\gamma^{T}b + \sum_{i=1}^{n} (\alpha_{i}^{k}-(\alpha_{i}^{*})^{k})y_{i} +\epsilon^{\text{opt}}C\nu \\
        & + \sum_{i=1}^{n}\alpha_{i}^{k}\xi_{i}^{\text{opt}}+(\alpha_{i}^{*})^{k}(\xi_{i}^{*})^{\text{opt}}-\mu^{T}d +\frac{1}{2} ||\beta^{\text{opt}}||^{2}. 
    \end{align*}

    The linear term that we wrote $l$ in the objective function of \cref{dual_polyhedral_svr} defines $l^{T}\theta = \sum_{i=1}^{n}(\alpha_{i}-\alpha_{i}^{*})X_{i:}+\gamma^{T}b-\mu^{T}d$ which in combination with the equality \cref{equality_dual_primal} gives
    \begin{align*}
        \frac{1}{2}||\beta^{k}-\beta^{\text{opt}}||^{2} & \leq \frac{1}{2} f(\theta^{k})  +\epsilon^{\text{opt}}C\nu 
     + \sum_{i=1}^{n}\alpha_{i}^{k}\xi_{i}^{\text{opt}}+(\alpha_{i}^{*})^{k}(\xi_{i}^{*})^{\text{opt}} +\frac{1}{2} ||\beta^{\text{opt}}||^{2}.
    \end{align*}

    Each $\alpha_{i}^{k}$, $(\alpha_{i}^{*})^{k}$ is bounded by $\frac{C}{n}$ which yields to 
    \begin{align*}
        \frac{1}{2}||\beta^{k}-\beta^{\text{opt}}||^{2} &\leq f(\theta^{k}) +\epsilon^{\text{opt}}C\nu + \frac{C}{n}\sum_{i=1}^{n}\xi_{i}^{\text{opt}}
        +(\xi_{i}^{*})^{\text{opt}}
        +\frac{1}{2} ||\beta^{\text{opt}}||^{2}.
    \end{align*}

    We recognize the objective function of the primal optimization problem and using that there is no dual gap at the optimum it follows that 
    $$\epsilon^{\text{opt}}C\nu + \frac{C}{n}\sum_{i=1}^{n}\xi_{i}^{\text{opt}}
    +(\xi_{i}^{*})^{\text{opt}}
    +\frac{1}{2} ||\beta^{\text{opt}}||^{2} = -f(\theta^{\text{opt}}),$$
    which finishes the proof.
    \hspace{1em}
\end{proof}

 Before the next statement, we need to give a definition that we will use in the next proofs. 

\begin{definition}
    Let $(i,j)$ ($i\in I_{\text{low}}$ and $j\in I_{\text{up}}$) be the most violating pair of variables in the block $\alpha$, $(i^{*},j^{*})$ ($i^{*}\in I^{*}_{\text{low}}$ and $j^{*}\in I^{*}_{\text{up}}$) for the block $\alpha^{*}$. Let $s_{1}$ be the index of the most violating variable in the block $\gamma$ and $s_{2}$ in the block $\mu$. We will call "optimality score" at iteration $k$ the quantity $\Delta^{k} = \max (\Delta_{1}^{k}, \Delta_{2}^{k}, \Delta_{3}^{k}, \Delta_{4}^{k}),  $ where $\Delta_{1}^{k} = \max (\nabla_{\alpha_{j}} f(\theta^{k}) - \nabla_{\alpha_{i}} f(\theta^{k}), 0)$, $\Delta_{2}^{k} = \max(\nabla_{\alpha_{j^{*}}} f(\theta^{k}) - \nabla_{\alpha_{i^{*}}} f(\theta^{k}),0)$, $\Delta_{3}^{k} = \max (- \nabla_{\gamma_{s_{1}}} f(\theta^{k}), 0)$ and $\Delta_{4}^{k} = \max (|\nabla_{\mu_{s_{2}}} f(\theta^{k})|, 0)$.
\end{definition}
The next result states that the sequence $\{f(\theta^{k})\}$ is a decreasing sequence. This result already states the convergence to a certain value $\bar{f}$ because we know that the sequence is bounded by the existing global minimum of the function since $f$ is convex.
\begin{lemma}
    \label{lemma_decrease}
    The sequence generated by the Generalized SMO algorithm $\{ f(\theta^{k})\}$ is a decreasing sequence. This sequence converges to a value $\bar{f}$. 
\end{lemma}

\begin{proof}
    We first prove that $f(\theta^{k})-f(\theta^{k+1})\geq 0$ when minimization takes place in the block $\alpha$. Let $(i,j)$ be the indices of the variables selected to be optimized and let $u\in\mathbb{R}^{2n+k_{1}+k_{2}}$ be the vector with only zeros except at the $i^{th}$ coordinate where it is equal to $t^{*}$ as defined in \cref{def_updates} and at the $j^{th}$ coordinate where it is equal to $-t^{*}$. We will also define $t_{q}=\frac{-(\nabla_{\alpha_{i}}f(\theta^{k})-\nabla_{\alpha_{j}}f(\theta^{k}))}{Q_{ii}+Q_{jj}-2Q_{ij}}$, the unconstrained minimum for the update in $\alpha$ block. Let us compute
    \begin{align*} 
        f(\theta^{k})-f(\theta^{k+1}) & = \frac{1}{2}(\theta^{k})^{T}\bar{Q}\theta^{k}+l^{T}\theta^{k}-\frac{1}{2}(\theta^{k+1})^{T}\bar{Q}\theta^{k+1}+l^{T}\theta^{k+1} \\
        & =\frac{1}{2}(\theta^{k})^{T}\bar{Q}\theta^{k}+l^{T}\theta^{k}-\frac{1}{2}(\theta^{k}+U)^{T}\bar{Q}(\theta^{k}+u)+l^{T}(\theta^{k}+u) \\
        & = -\frac{1}{2}u^{T}\bar{Q}u-u^{T}(Q\theta^{k}+l) \\
        &=-\frac{1}{2}u^{T}\bar{Q}u-u^{T}(\nabla f(\theta^{k})) \\
        &= -\frac{(t^{*})^{2}}{2}(Q_{ii}+Q_{jj}-2Q_{ij})-t^{*}(\nabla_{\alpha_{i}}f(\theta^{k})-\nabla_{\alpha_{j}}f(\theta^{k})).
    \end{align*}

    We first study the case when there is no clipping which means that $t^{*}=t_{q}$

    \textbf{1. No clipping}. Replacing $t^{*}$ by its expression leads to the following result: 
    \begin{align*}
        f(\theta^{k})-f(\theta^{k+1}) &= \frac{(\Delta^{k}_{1})^{2}}{2(Q_{ii}+Q_{jj}-2Q_{ij})} \\
        & = \frac{(\Delta^{k}_{1})^{2}}{2||X_{i:}-X_{j:}||^{2}}\geq 0.
    \end{align*}

    \textbf{2. Clipping takes place because }$t_{q}\leq t^{*}=\max(-\alpha_{i},\alpha_{j}-\frac{C}{n})$\\
    We notice that $t_{q}\leq \max(-\alpha_{i},\alpha_{j}-\frac{C}{n})\leq 0$ which implies that $i\in I_{\text{low}}$ and $j\in I_{\text{up}}$. In that case $\Delta_{1}^{k} = \nabla_{\alpha_{i}} f(\theta^{k})-\nabla_{\alpha_{j}} f(\theta^{k})$. Replacing $t_{q}$ by its expression leads to 
    \begin{align*}
        -(\nabla_{\alpha_{i}}f(\theta^{k})-\nabla_{\alpha_{j}}f(\theta^{k})) & \leq t^{*}(Q_{ii}+Q_{jj}-2Q_{ij}) \\
        \frac{\Delta^{k}_{1}t^{*}}{2} & \leq \frac{-(t^{*})^{2}}{2}(Q_{ii}+Q_{jj}-2Q_{ij}) \\
        \frac{\Delta^{k}_{1}t^{*}}{2} -t^{*}\Delta^{k}_{1} & \leq \frac{-(t^{*})^{2}}{2}(Q_{ii}+Q_{jj}-2Q_{ij})-t^{*}(\nabla_{\alpha_{i}}f(\theta^{k})-\nabla_{\alpha_{j}}f(\theta^{k})) \\
        -\frac{1}{2}\Delta^{k}_{1}t^{*} & \leq \frac{-(t^{*})^{2}}{2}(Q_{ii}+Q_{jj}-2Q_{ij})-t^{*}(\nabla_{\alpha_{i}}f(\theta^{k})-\nabla_{\alpha_{j}}f(\theta^{k}))
    \end{align*}
    Thus we have that if $t^{*}=-\alpha_{i}$,
    $$
    f(\theta^{k})-f(\theta^{k+1})\geq \frac{1}{2}\Delta^{k}_{1}\alpha_{i}\geq 0
    $$
    and that if $t^{*}=\alpha_{j}-\frac{C}{n}$,
    $$
    f(\theta^{k})-f(\theta^{k+1})\geq \frac{1}{2}\Delta^{k}_{1}(\frac{C}{n}-\alpha_{j})\geq 0.
    $$
    \textbf{3. Clipping takes place because } $t_{q}\geq t^{*}=\min(\frac{C}{n}-\alpha_{i},\alpha_{j})$.\\
    This time $t_{q}\geq \min(\frac{C}{n}-\alpha_{i},\alpha_{j})\geq 0$ which also implies that $i\in I_{\text{up}}$ and $j\in I_{\text{low}}$ and that $\Delta_{1}^{k} = \nabla_{\alpha_{j}} f(\theta^{k})-\nabla_{\alpha_{i}} f(\theta^{k})$. The only difference here is that multiplying by $-t^{*}$ will imply a change in the inequality.

    \begin{align*}
        -(\nabla_{i}f(\theta^{k})-\nabla_{j}f(\theta^{k})) & \geq t^{*}(Q_{ii}+Q_{jj}-2Q_{ij}) \\
        \frac{-\Delta_{1}^{k}t^{*}}{2} & \leq \frac{-(t^{*})^{2}}{2}(Q_{ii}+Q_{jj}-2Q_{ij}) \\
        \frac{-\Delta_{1}^{k}t^{*}}{2} +t^{*}\Delta_{1}^{k} & \leq \frac{-(t^{*})^{2}}{2}(Q_{ii}+Q_{jj}-2Q_{ij})-t^{*}(\nabla_{i}f(\theta^{k})-\nabla_{j}f(\theta^{k})) \\
        \frac{1}{2}\Delta_{1}^{k}t^{*} & \leq \frac{-(t^{*})^{2}}{2}(Q_{ii}+Q_{jj}-2Q_{ij})-t^{*}(\nabla_{i}f(\theta^{k})-\nabla_{j}f(\theta^{k}))
    \end{align*}

    Thus we have that if $t^{*}= \frac{C}{n}-\alpha_{i}$
    $$
    f(\theta^{k})-f(\theta^{k+1})\geq \frac{1}{2}\Delta^{k}_{1}(\frac{C}{n}-\alpha_{i})\geq 0, 
    $$
    and if $t^{*}= \alpha_{j}$,
    $$
    f(\theta^{k})-f(\theta^{k+1})\geq \frac{1}{2}\Delta^{k}_{1}\alpha_{j}\geq 0. 
    $$
    To prove that $f(\theta^{k})-f(\theta^{k+1})\geq 0$ when the update takes place in the block $\gamma$ and $\mu$ we first need to observe that 
    when only one variable is updated between iteration $k$ and $k+1$ it follows that
    $$
    f(\theta^{k})-f(\theta^{k+1})=\frac{1}{2}\bar{Q}_{ii}(\theta^{k}_{i}-\theta_{i}^{k+1})^{2}.
    $$

    Therefore, we now prove the result for the block $\gamma$. If the update is not a clipped update and $i$ is the index of the updated variable, it holds that
    $$\gamma^{k}_{i}-\gamma_{i}^{k+1}=\frac{\nabla_{\gamma_{i}}f(\theta^{k})}{(AA^{T})_{ii}},$$ which gives the following bound
    \begin{equation}
        \label{diff_obj}
        f(\theta^{k})-f(\theta^{k+1})=\frac{1}{2(AA^{T})_{ii}}(\nabla_{\gamma_{i}}f(\theta^{k}))^{2}\geq 0.
    \end{equation}
    Moreover, if a clipped update takes place in this block, we know that it happens when $0\leq \gamma^{k}_{i}\leq \frac{\nabla_{\gamma_{i}}f(\theta^{k})}{(AA^{T})_{ii}}$. It yields to the following bound
    $$
    f(\theta^{k})-f(\theta^{k+1})=\frac{1}{2}(AA^{T})_{ii}(\gamma^{k}_{i})^{2}\geq 0.
    $$
     The result for the block $\mu$ is obtained using the same arguments except that there is no clipped updates. 
\end{proof}

\begin{lemma}
    There exists a subsequence $\{\theta^{k_{j}}\}$ of iterations generated by the generalized SMO where clipping does not take place.
\end{lemma}

\begin{proof}
    Let's suppose the contrary, which means that there exists an iteration $K$ such that for all $k\geq K$ we only perform clipped updates. The number of variables $N_{B}^{k}$ that belong to the boundary of its contraints ($0$ or $\frac{C}{n}$ for the blocks $\alpha$ or $\alpha^{*}$ and $0$ for the block $\gamma$) is non-decreasing for all $k\geq K$ and it is bounded thus it must converge to another integer $N^{*}$. 
    
    This convergence implies that there exists $k^{*}$ such that for all $k\geq k^{*}$, $N_{B}^{k}=N^{*}$ since $N_{B}^{k}$ and $N^{*}$ are integers. This observation allows us to conclude that for all $k\geq k^{*}$ clipped updates only take place in the blocks $\alpha$ or $\alpha^{*}$ since the updates in the block $\gamma$ are made on only one variable and that the number of clipped variables has reached its maximum value. An update in the block $\gamma$ would strictly increase the number of clipped variables which is not possible for all $k\geq k^{*}$ or the update would not change the value of $\theta$ and we showed before that this situation is not possible (\cref{prop_non_violating}). 
    
    For all $k\geq k^{*}$, we have that updates in the block $\alpha$ (resp. $\alpha^{*}$) have this necessary scheme: $\alpha_{i}^{k}$ or $\alpha_{j}^{k}$ is equal to $0$ or $\frac{C}{n}$ thus after the update, one of them will leave the boundary and the other one goes to it in order to keep the number of clipped variables equals to $N^{*}$. The different possibilities are then the following:
    \begin{itemize}
        \item if $\alpha_{i}^{k}=0$ and $0<\alpha_{j}^{k}\leq \frac{C}{l}$ the only possible update following the \cref{def_updates} is
        $$ \alpha_{i}^{k+1}=\alpha_{i}^{k}+\alpha_{j}^{k}=\alpha_{j}^{k}$$
        $$
        \alpha_{j}^{k+1}=\alpha_{j}^{k}-\alpha_{j}^{k}=0.
        $$
        \item if $\alpha_{j}^{k}=\frac{C}{l}$ and $0\leq\alpha_{i}^{k}< \frac{C}{l}$ the only possible update following the \cref{def_updates} is
        $$ \alpha_{i}^{k+1}=\alpha_{i}^{k}+(\frac{C}{l}-\alpha_{i}^{k})=\frac{C}{l}$$
        $$
        \alpha_{j}^{k+1}=\alpha_{j}^{k}-(\frac{C}{l}-\alpha_{i}^{k})=\alpha_{i}^{k}.
        $$
    \end{itemize}
    It stays true for the block $\alpha^{*}$ and the discussion is similar.  
    It is clear that from the description of the updates made above that there is only a finite number of ways to shuffle the values which means that there exists $k_{1},k_{2} \geq k^{*}$ such as $\theta^{k_{1}}=\theta^{k_{2}}$ and with $k_{1}<k_{2}$. Therefore $f(\theta^{k_{1}})=f(\theta^{k_{2}})$ which contradicts the decrease of the sequence $f(\theta^{k})$ (\cref{lemma_decrease}).

\end{proof}

\begin{lemma}
    \label{delta_convergence}
    Let $\{\theta^{k_{j}}\}$ be a subsequence generated by the Generalized SMO algorithm where clipping does not take place. We then have that $\Delta^{k_{j}} \rightarrow 0$.
\end{lemma}
\begin{proof}
    We have that $f(\theta^{k_{j}})-f(\theta^{k_{j}+1})\geq \frac{(\nabla_{\alpha_{i}}f(\theta^{k_{j}})-\nabla_{\alpha_{j}}f(\theta^{k_{j}}))^{2}}{2D^{2}}= \frac{(\Delta^{k_{j}})^{2}}{2D^{2}}$ where $D=\underset{p,q}{\max} || X_{p:}-X_{q:}||$ when the update happens in the blocks $\alpha$ or $\alpha^{*}$. When it happens in the block $\gamma$ with no clipping we have the following inequality
    $f(\theta^{k_{j}})-f(\theta^{k_{j}+1})\geq \frac{(\nabla_{\gamma_{i}}f(\theta^{k_{j}}))^{2}}{2}=\frac{(-\Delta^{k_{j}})^{2}}{2} = \frac{(\Delta^{k_{j}})^{2}}{2}$. When the update takes place in the block $\mu$, we have that $f(\theta^{k_{j}})-f(\theta^{k_{j}+1})\geq \frac{(\nabla_{\mu_{i}}f(\theta^{k_{j}}))^{2}}{2})= \frac{(\Delta^{k_{j}})^{2}}{2}$. We then define a sequence   \[   
        u^{k_{j}} = 
             \begin{cases}
              \frac{1}{2D^{2}}(\Delta^{k_{j}})^{2} &\quad\text{if the update takes place in the blocks }\alpha \text{ or } \alpha^{*}.\\
              \frac{1}{2}(\Delta^{k_{j}})^{2} &\quad\text{if the update takes place in the blocks }\gamma \text{ or } \mu.\\
             \end{cases}
        \]

        The sequence $\{u^{k_{j}}\}\rightarrow 0$ because of the bound given above and the fact that $f(\theta^{k_{j}})-f(\theta^{k_{j}+1})\rightarrow 0$ too (\cref{lemma_decrease}). This implies that $\Delta^{k_{j}}\rightarrow 0$ as well. 
\end{proof}
A consequence of the lemma above is that $\Delta_{1}^{k_{j}}\rightarrow 0$, $\Delta_{2}^{k_{j}}\rightarrow 0$, $\Delta_{3}^{k_{j}}\rightarrow 0$ and $\Delta_{4}^{k_{j}}\rightarrow 0$ because $\Delta^{k_{j}}$ is defined as the maximum of those four positive values.
\begin{lemma}
    \label{lemma_bounded}
    Let $\{\theta^{k_{j}}\}$ be a subsequence generated by the generalized SMO algorithm where clipping does not take place. This subsequence is bounded.
\end{lemma}

\begin{proof}
    To prove the statement, we will show that $||\theta^{k_{j}}-\theta^{\text{opt}}||^{2}$ is bounded where $\theta^{\text{opt}}$ belongs to the set of solution of \cref{dual_polyhedral_svr}. Since each $\alpha_{i}$ and $\alpha_{i}^{*}$ is belongs to $[0,\frac{C}{n}]$, we have that
    \begin{align*}
        ||\theta^{k_{j}+1}-\theta^{\text{opt}}||^{2} & = ||\alpha^{k_{j}+1}-\alpha^{\text{opt}}||^{2} + ||(\alpha^{*})^{k_{j}+1}-(\alpha^{*})^{\text{opt}}||^{2} + ||\gamma^{k_{j}+1}-\gamma^{\text{opt}}||^{2}\\ & + ||\mu^{k_{j}+1}-\mu^{\text{opt}}||^{2} \\
         & \leq \frac{2C^{2}}{n} + ||\gamma^{k_{j}+1}-\gamma^{\text{opt}}||^{2} + ||\mu^{k_{j}+1}-\mu^{\text{opt}}||^{2}.
    \end{align*}
    
    We will work on the bound for the quantity $||\mu^{k_{j}+1}-\mu^{\text{opt}}||^{2}$ first. If the update happens in the block $\mu$ at coordinate $\mu_{j}$, we have the following 
    \begin{align*}
    ||\mu^{k_{j}+1}-\mu^{\text{opt}}||^{2} & =  ||\mu^{k_{j}}-e_{j}\frac{\nabla_{\mu_{j}}f(\theta^{k_{j}})}{(\Gamma\Gamma^{T})_{jj}}-\mu^{\text{opt}}||^{2} \\
    & = ||\mu^{k_{j}}-\mu^{\text{opt}}||^{2} -2 \langle \mu^{k_{j}}-\mu^{\text{opt}}, e_{j}\frac{\nabla_{\mu_{j}}f(\theta^{k_{j}})}{(\Gamma\Gamma^{T})_{jj}} \rangle + ||e_{j}\frac{\nabla_{\mu_{j}}f(\theta^{k_{j}})}{(\Gamma\Gamma^{T})_{jj}} ||^{2}  \\
    & = ||\mu^{k_{j}}-\mu^{\text{opt}}||^{2} + \frac{\nabla_{\mu_{j}}f(\theta^{k_{j}})^{2}}{(\Gamma\Gamma^{T})_{jj}^{2}}-2 \frac{\nabla_{\mu_{j}}f(\theta^{k_{j}})}{(\Gamma\Gamma^{T})_{jj}}(\mu_{j}^{k_{j}}-\mu^{\text{opt}}_{j}).
    \end{align*}

    We then have that
    \begin{align*}
        -2 \frac{\nabla_{\mu_{j}}f(\theta^{k_{j}})}{(\Gamma\Gamma^{T})_{jj}}(\mu_{j}^{k_{j}}-\mu^{\text{opt}}_{j}) & = 2 (\mu^{k_{j}+1}_{j}-\mu^{k_{j}}_{j})(\mu_{j}^{k_{j}}-\mu^{\text{opt}}_{j}) \\
        & = 2\langle \mu^{k_{j}+1}-\mu^{k_{j}}, \mu^{k_{j}}-\mu^{\text{opt}} \rangle \\
        & \leq 2|| \mu^{k_{j}+1}-\mu^{k_{j}} || \cdot || \mu^{k_{j}}-\mu^{\text{opt}} || \\
        & \leq 2 \frac{|\nabla_{\mu_{j}}f(\theta^{k_{j}})|}{(\Gamma\Gamma^{T})_{jj}} || \mu^{k_{j}}-\mu^{\text{opt}} || \\
        & \leq 2 \frac{\Delta_{4}^{k_{j}}}{(\Gamma\Gamma^{T})_{jj}}|| \mu^{k_{j}}-\mu^{\text{opt}} || \\
    \end{align*}
    From \cref{delta_convergence}, we have that $\Delta_{4}^{k_{j}}\rightarrow 0$ then it can be bounded by a constant $M_{0}$. 
    We know from \cref{diff_obj} that $\frac{\nabla_{\mu_{j}}f(\theta^{k_{j}})^{2}}{(\Gamma\Gamma^{T})_{jj}^{2}} = \frac{2}{(\Gamma\Gamma^{T})_{jj}}(f(\theta^{k_{j}})-f(\theta^{k_{j}+1}))$. From \cref{lemma_decrease}, we know that $f(\theta^{k_{j}})-f(\theta^{k_{j}+1})\rightarrow 0$ then it can be bounded by a constant $M_{1}$. Overall we have that 
    \begin{equation*}
        ||\mu^{k_{j}+1}-\mu^{\text{opt}}||^{2}  \leq ||\mu^{k_{j}}-\mu^{\text{opt}}||^{2} + 2 \frac{M_{0}}{(\Gamma\Gamma^{T})_{jj}}|| \mu^{k_{j}}-\mu^{\text{opt}} || + \frac{2}{(\Gamma\Gamma^{T})_{jj}}M_{1}.
    \end{equation*}

    By recursion we have
    $$
    ||\mu^{k_{j}+1}-\mu^{\text{opt}}||^{2}  \leq ||\mu^{0}-\mu^{\text{opt}}||^{2} + 2 \frac{M_{0}}{(\Gamma\Gamma^{T})_{jj}}|| \mu^{0}-\mu^{\text{opt}} || + \frac{2}{(\Gamma\Gamma^{T})_{jj}}M_{1}<\infty.
    $$ 

    Since there is no clipped update on the subsequence $\{\theta^{k_{j}}\}$, the proof for the block $\gamma$ is similar which proves that $||\theta^{k_{j}}-\theta^{\text{opt}}||$ is bounded. 
\end{proof}

\begin{lemma}
    Let $\{\theta^{k_{j}}\}$ be a subsequence generated by the generalized SMO algorithm where clipping does not take place. There exists a sub-subsequence that converges to $\bar{\theta}$, with $\bar{\theta}$ being a solution of $\cref{dual_polyhedral_svr}$.   
\end{lemma}

\begin{proof}
    From \cref{lemma_bounded}, we have that $\{\theta^{k_{j}}\}$ is a bounded sequence, it means that we can extract a converging subsequence that we will write $\{\theta^{k_{j}}\}$ not to complicate the notations. 
    Since $\mathcal{F}$ is closed, $\bar{\theta}$ meets the constraints of the dual optimization problem and belongs to $\mathcal{F}$. We now want to prove that it belongs to the set of solution of \cref{dual_polyhedral_svr} by showing that $\bar{\Delta}_{1}(\bar{\theta})\leq 0$, $\bar{\Delta}_{2}(\bar{\theta})\leq 0$, $\bar{\Delta}_{3}(\bar{\theta})\leq 0$ and $\bar{\Delta}_{4}(\bar{\theta})\leq 0$. Let's make two observations that will be used for the following proof.
    The first one comes from the continuity of the gradient which implies that for all $\epsilon$ there exists $K_{1}$ such that for all $k_{j}\geq K_{1}$,  $|\nabla_{i}f(\theta^{k_{j}})-\nabla_{i}f(\bar{\theta})|<\epsilon$ for all $i$. 
    The second observation is that it is possible too chose an $\epsilon$ small enough such that there exists  $K_{2}$ such that for all $k_{j}\geq K_{2}$: if $\bar{\alpha_{i}}>0$, we have $\alpha_{i}^{k_{j}}>0$ and if $\bar{\alpha_{i}}<\frac{C}{n}$ we have $\alpha_{i}^{k_{j}}<\frac{C}{n}$. In other words, we say that all the indices in the set $I_{\text{low}}(\bar{\alpha}) (\text{ resp. }I_{\text{up}})$ are also in $I_{\text{low}}(\alpha_{k_{j}}) (\text{ resp. }I_{\text{\text{up}}})$. The same argument holds for indices in the block $\alpha^{*}$. 

    Let's assume that $\bar{\Delta}_{1}>0$, it means that there exists at least one violating pair of variables that we will note $(\bar{i},\bar{j})$ at $\bar{\theta}$. From the discussion above, we know that $\bar{i}\in I_{\text{low}}$ for all $k_{j}\geq K_{2}$ and that $\bar{j}\in I_{\text{up}}$ for all $k_{j}\geq K_{2}$. We then have that for all $\epsilon>0$, there exists $K_{1}$ such as for all $k_{j}\geq \max(K_{1}, K_{2})$,
    \begin{align*}
        \Delta_{1}^{k_{j}} &=\underset{i\in I_{\text{up}}}{\min}\nabla_{i}f(\theta^{k_{j}})-\underset{i\in I_{\text{low}}}{\max}\nabla_{i}f(\theta^{k_{j}})\\
         & \geq \nabla_{\bar{i}}f(\theta^{k_{j}})-\nabla_{\bar{j}}f(\theta^{k_{j}}) \\
         & \geq (\nabla_{\bar{i}}f(\bar{\theta})-\epsilon)-(\nabla_{\bar{j}}f(\bar{\theta})+\epsilon) \\
         &= \bar{\Delta}_{1}-2\epsilon.  
        \end{align*}   
        
        We choose $\epsilon= \frac{\bar{\Delta}_{1}}{2}-\epsilon'$ where $0<\epsilon'<\frac{\bar{\Delta}_{1}}{2}$ which leads to :
        $$
        \Delta_{1}^{k_{j}}\geq \bar{\Delta}_{1}-2\epsilon'=2\epsilon' >0.
        $$
        This inequality is true for all $k_{j}\geq \max(K_{1},K_{2})$ which contradicts the fact that $\Delta_{1}^{k_{j}}\rightarrow 0$. 
        The proof is similar to show that $\bar{\Delta}_{2}\leq 0$. 

        Let's now suppose that $\bar{\Delta}_{3}> 0$ it means that there exists an index $\bar{i}$ such that $\nabla_{\gamma_{i}}f(\bar{\theta})<0$. For all $\epsilon>0$, there exists $K_{1}$, $K_{2}$ such as for all $k_{j}>\max(K_{1}, K_{2})$ 
        \begin{align*}
            \Delta_{3}^{k_{j}} &=-\underset{i\in \{1,\hdots,k_{1}\}}{\min}\nabla_{\gamma_{i}}f(\theta^{k_{j}})\\
             & \geq -\nabla_{\gamma_{\bar{i}}}f(\theta^{k_{j}}) \\
             & \geq -(\nabla_{\gamma_{\bar{i}}}f(\bar{\theta})+\epsilon) \\
             &= \bar{\Delta}_{3}-\epsilon.
            \end{align*}

            We choose $\epsilon= \bar{\Delta}_{3}-\epsilon'$ where $0<\epsilon'<\bar{\Delta}_{3}$ which leads to :
            $$
            \Delta_{3}^{k_{j}}\geq \bar{\Delta}_{3}-\epsilon=\epsilon' >0.
            $$
            This inequality is true for all $k_{j}\geq \max(K_{1},K_{2})$ which contradicts the fact that $\Delta_{3}^{k_{j}}\rightarrow 0$.

            Finally let's assume that $\Delta_{4}^{k_{j}}>0$, it means that $|\nabla_{\mu_{i}}f(\theta^{k_{j}})|\neq 0$.  
            Using the continuity of the gradient we write that for all $\epsilon>0$ there exists $K_{1}$ such that for all $k_{j}\geq K_{1}$ we have $|\nabla_{\mu_{i}}f(\theta^{k_{j}})-\nabla_{\mu_{i}}f(\bar{\theta}|<\epsilon$. Using triangle inequality we get that 
            $$
            \bigg | |\nabla_{\mu_{i}}f(\theta^{k_{j}})|-|\nabla_{\mu_{i}}f(\bar{\theta})|\bigg |\leq|\nabla_{\mu_{i}}f(\theta^{k_{j}})-\nabla_{\mu_{i}}f(\bar{\theta})|<\epsilon.
            $$
            Thus
            $$
            -\epsilon\leq |\nabla_{\mu_{i}}f(\theta^{k_{j}})|-|\nabla_{\mu_{i}}f(\bar{\theta})| \leq \epsilon, 
            $$

            which means that 
            $$
            |\nabla_{\mu_{i}}f(\bar{\theta})|-\epsilon\leq |\nabla_{\mu_{i}}f(\theta^{k_{j}})|.
            $$
            Then we have the following: 
            \begin{align*}
                \Delta_{4}^{k_{j}} &=\underset{i\in\{1,\hdots, k_{2}\}}{\max}|\nabla_{\mu_{i}}f(\theta^{k_{j}})|\\
                & \geq \nabla_{\mu_{\bar{i}}}f(\theta^{k_{j}})| \\
                 & \geq |\nabla_{\sigma}f(\bar{\theta})|-\epsilon \\
                 &= \bar{\Delta}_{4}-\epsilon.
                \end{align*}

            We choose $\epsilon= \bar{\Delta}_{4}-\epsilon'$ where $0<\epsilon'<\bar{\Delta}_{4}$ which leads to           
            $$ 
            \Delta_{4}^{k_{j}}\geq \bar{\Delta}_{4}-\epsilon=\epsilon' >0.
            $$
            This inequality is true for all $k_{j}\geq \max(K_{1},K_{2})$ which contradicts the fact that $\Delta_{4}^{k_{j}}\rightarrow 0$. 
\end{proof}

\begin{proof}
    We are now able to give the proof of the \cref{convergence_theorem}.
From \cref{lemma_bound}, we have that $\frac{1}{2}||\beta^{k}-\beta^{\text{opt}}||\leq f(\theta^{k})-f(\theta^{\text{opt}}).$ Moreover, from \cref{delta_convergence} we know that there is a subsequence $\{\theta^{k_{j}}\}$ generated by the Generalized SMO algorithm where clipping does not take place and that converges to $\bar{\theta}$, with $\bar{\theta}$ a solution of \cref{dual_polyhedral_svr}. 
The continuity of the objective function $f$ allows us to say that $f(\theta^{k_{j}})\rightarrow f(\bar{\theta})$. 
From \cref{lemma_decrease}, we know that $\{f(\theta^{k_{j}})\}$ is decreasing and bounded so the monotone convergence theorem implies that the whole sequence $f(\theta^{k})\rightarrow f(\bar{\theta})$ and it follows that $\frac{1}{2}||\beta^{k}-\beta^{\text{opt}}||\rightarrow 0$ and finally that $\beta^{k}\rightarrow \beta^{\text{opt}}.$ 
\end{proof}
\end{document}